\documentclass[a4paper, 11pt]{article}

\usepackage{amsmath}
\usepackage{amsfonts}
\usepackage{amssymb}
\usepackage[english]{babel}
\usepackage{graphicx}
\usepackage{amsthm}
\usepackage{enumerate}
\usepackage{tikz}

\newtheorem{proposition}{Proposition}
\newtheorem{theorem}{Theorem}
\newtheorem{lemma}{Lemma}

\newtheorem*{cor4.1}{Corollary 4.1}
\newtheorem*{th2.5}{Theorem 2.5}

\theoremstyle{definition}

\date{}

\begin{document}

\title{Quaternionic $1-$factorizations and complete sets of rainbow spanning trees}

\author{
G. Rinaldi\thanks{Dipartimento di Scienze e Metodi dell'Ingegneria,
Universit\`a di Modena e Reggio Emilia, via Amendola 2, 42122 Reggio
Emilia (Italy) gloria.rinaldi@unimore.it  Research performed within the activity of INdAM--GNSAGA.}}

\author{
G. Rinaldi\thanks{Dipartimento di Scienze e Metodi dell'Ingegneria,
Universit\`a di Modena e Reggio Emilia, via Amendola 2, 42122 Reggio
Emilia (Italy) gloria.rinaldi@unimore.it  Research performed within the activity of INdAM--GNSAGA.}}

\maketitle

\begin{abstract}
\noindent
\noindent
A $1-$factorization ${\cal F}$  of a complete graph $K_{2n}$ is said to be $G-$regular, or regular under $G$, if $G$ is an automorphism group of ${\cal F}$ acting sharply transitively on the vertex-set. The problem of determining which groups can realize such a situation dates back to
a result by Hartman and Rosa (1985) on cyclic groups and, when $n$ is even, the problem  is still open,  even though several classes of groups were tested in the recent past.
It was recently proved, see Rinaldi (2021) and Mazzuoccolo et al. (2019), that a $G-$regular $1-$factorization together with  a complete set of rainbow spanning trees exists whenever $n$ is odd, while the existence for $n$ even was proved when either $G$ is cyclic and $n$ is not a power of $2$,  or when $G$ is a dihedral group. In this paper we extend this result and  prove the existence also for the following classes of groups:  Abelian but not cyclic, dicyclic, non cyclic $2-$groups with a cyclic subgroup of index $2$.
\end{abstract}

\noindent \textit{Keywords: Regular $1$-factorizations, complete graph, sharply transitive permutation groups, starter, rainbow spanning trees.}

\noindent\textit{MSC(2010): 05C70-05C15-05C05-05C51}

\section{Introduction}\label{sec:intro}

It is well known that the number of non-isomorphic $1-$factorizations
of $K_{2n}$, the complete graph on $2n$ vertices,
goes to infinity with the positive integer $n$, \cite{C}. Therefore, attempts to achieve classifications
can be done if one imposes additional conditions either on the $1-$factorization or on its automorphism
group. For example, a precise description of the $1-$factorization
and of its automorphism group was given when the group is assumed to act multiply
transitively on the vertex set, \cite{CK}.

Few years ago the following question was adressed:

\vspace{0.2cm}\noindent
{\bf Question.} {\em Let $G$ be a group of order $2n$. Does there exist a $1-$factorization of
$K_{2n}$ admitting $G$ as an automorphism group acting sharply
transitively on the vertex-set of $K_{2n}$?}

\vspace{0.2cm}\noindent
A $1-$factorization of $K_{2n}$ satisfying the above condition is said to be
$G$-{\em regular} or {\em regular} under $G$.

This question is a restricted version of
problem n.4 in the list of \cite{W1}, namely the word ``sharply'' does not
appear there, but the two versions are equivalent for abelian groups, since every
transitive abelian permutation group is sharply transitive.
When $n$ is odd the problem simplifies somewhat: $G$ must be the
semi-direct product of $Z_2$ with its normal complement and $G$
always realizes a $1-$factorization of $K_{2n}$ upon which it acts
sharply transitively on vertices, see \cite[Remark 1]{BL}.
When $n$ is even, the complete answer is still unknown.

If $G$ is a cyclic group then Hartman and Rosa proved in \cite{HR}
that the answer to the above question is negative when $n$ is a power of $2$
greater than $2$, while it is affirmative for all other values of $n$.
In a most recent past, an affirmative answer was given for several other classes of groups,
see for example \cite{Bu},  \cite{BL}, \cite{BR}, \cite{R1}
which respectively consider the class of abelian, dihedral, dicyclic and other nilpotente groups.
In  \cite{Bonv2} and  \cite{PP} a positive answer was found for the class of $2-$groups with an elementary
abelian Frattini subgroup and for some  non-solvable groups, respectively.
Also, nonexistence results were achieved
by assuming the existence of a fixed $1-$factor, \cite{K}, \cite{R1}.
Further results were obtained
when the number of fixed $1-$factors is as large as possible, \cite{BL}, or when the $1-$factors
satisfy some additional requests, \cite{Bonv1}. Recently, we focused our attention on the existence
of $G$-regular $1-$factorizations of $K_{2n}$
which possess a complete set of rainbow spanning trees, \cite{MR},\cite{R2}.

We recall that a rainbow spanning tree is a spanning tree sharing exactly
one edge with each $1-$factor of the given $1-$factorization.
In other words, a $1-$factorization of $K_{2n}$ corresponds to a proper
edge coloring of $K_{2n}$ with precisely $2n-1$ colors: each color appears exactly $n$ times
and corresponds to a $1-$factor.
Therefore, a spanning tree is {\it rainbow} if its edges have distinct colors.
It is also usual to say that such a tree is {\it orthogonal} to the $1-$factorization.
We also recall that if $T$ is any subgraph of $K_{2n}$ with exactly $2n-1$ edges,
then $T$ is a spanning tree if and only if $T$ is a spanning connected graph,
see for instance \rm\cite[6, p.68]{W}.

A set of rainbow spanning trees is said to be {\it a complete set}
if the trees form a  partition of the edge set of $K_{2n}$.
It is easy to prove that a complete set cannot exist in $K_4$,
so we restrict our discussion to complete sets in $K_{2n}$ with $n\ge 3$.
Also, since each rainbow spanning tree has $2n-1$ edges,
$n$ is the number of disjoint trees in a complete set.

In \cite{R2} it is proved that, regardless of the isomorphism type of $G$,
a $G-$regular $1-$factorization of $K_{2n}$ together with a rainbow spanning tree,
whose orbit under a subgroup of $G$ gives rise to a complete set, exists
if and only if $n\ge 3$ is an odd number.
The problem of determining for which groups $G$ a $G-$regular
$1-$factorization, together with a complete set of rainbow spanning trees, exists remains open when the order of $G$ is twice an even number.
With some exceptions:
in \cite{MR} a complete set of rainbow spanning trees was constructed in the family of cyclic regular
$1-$factorizations of \cite{HR} for each $n\ge 3$, except when $n=2^s$, $s\ge 2$. In \cite{R2}
an explicit construction was given for the class of dihedral groups of order twice an even number.

Our main interest
fits in the general problem of characterizing $G-$regular $1-$factorizations satisfying additional properties.
However, I recall that the  problem of determining whether every given
$1-$factorization of a complete graph possesses
a complete set of rainbow spanning trees dates back to the Brualdi and Hollingsworth
conjecture, \cite{BH}, and to the Constantine conjecture when the trees are asked to be
pairwise isomorphic as uncolored trees, \cite{Co}. A recent asymptotic result settles both these
conjectures for all sufficiently large $n$, \cite{GKMH}.
Nevertheless, the solution for each given $n$ remains nontrivial even if one is allowed
to choose the $1-$factorization.

Most of the papers about these conjectures treat the general case by methods of extremal graph theory/probabilistic methods which can be applied for every $1$-factorization of $K_{2n}$. The best known results hold for large $n$ and mainly give lower-bounds on the number of rainbow spanning trees.
Together with \cite{GKMH} we recall some other important papers in this direction: \cite{AA}, \cite{FL}, \cite{H}, \cite{KMV}, \cite{MPS}, \cite{PS}.
The Brualdi-Hollingsworth conjecture was extended also in \cite{KKS}, by stating that edges of every properly colored $K_{n}$ (not necessarily colored by a $1$-factorization) can be partitioned into rainbow spanning trees.  Results are, for example, contained in \cite{BLM},  \cite{CHH}, \cite{MPS}, and for large $n$, the results of \cite{PS} improved the best known bounds for the three conjectures in \cite{BH}, \cite{Co} and \cite{KKS}.

Some examples of $1-$factorizations of $K_{2n}$ satisfying the above conjectures without imposing conditions on $n$
are also available. Constantine himself proved the existence of a
suitable $1-$factorization
satisfying his conjecture for the case $2n$ a power of $2$ or five times a power of two, \cite{Co}.

Also, a first family of $1-$factorizations for which the conjecture of Brualdi and Hollingsworth
can be verified for each $n\ge 3$ was recently shown in \cite{CKM}.

When $n$ is even, the examples of $G-$regular $1-$factorizations together
with a complete set of rainbow spanning trees obtained in \cite{MR} and \cite{R2},
involve groups possessing a cyclic subgroup of index $2$.
In the present paper we feel rather natural to try to extend the analysis in this direction.
More precisely, we consider dicyclic groups and abelian groups with a cyclic subgroup of index $2$.
We obviously exclude the family of cyclic $2-$groups, in fact
a regular $1-$factorization does not exist in these cases, \cite{HR}.
Moreover, we consider all the non-cyclic $2$-groups admitting a cyclic subgroup of index $2$.

The state of art can  be resumed in the following Theorem.

\begin{theorem}
Let $G$ be a group of order $2n$, $n > 2$ even. A $G-$regular $1-$factorization together with a complete set
of rainbow spanning trees exists whenever $G$ is one of the following:
a dihedral group; a dicyclic group; an abelian group admitting a cyclic subgroup of index $2$ and
different from a cyclic $2-$group;
a non-cyclic $2$-group admitting a cyclic subgroup of index $2$.
\end{theorem}

The dihedral case and the cyclic case were considered in \cite{R2} and \cite{MR}, respectively.
In the following sections \ref{dic},  \ref{2G}, \ref{Ab}, we will show explicit constructions
which will prove the existence in all the other cases.

For the sake of completeness, we recall that the finite non-abelian
$2$-groups (of order ${}\ge8$) admitting a cyclic
subgroup of index $2$ are known. Satz 14.9 in \cite{Hu} divides them into
four isomorphism types: $(1)$, $(2)$, $(3)$, $(4)$.  Groups
of type $(1)$ are dihedral groups,
while, each group $G$ of type
$(2)$, $(3)$ or $(4)$ is considered in this paper.

In this paper we will make use of the regular $1-$factorizations
already obtained in \cite{BR} and we refer to \cite{Bu} for the abelian case.
The $1-$factorizations constructed in \cite{BR}  were  referred to as {\it quaternionic $1-$factorizations},
since a type (2) group is a quaternionic one.
This inspired the title of the present paper.

\subsection{Preliminaries}\label{pre}
We refer to the monograph \cite{W} for the general notions on graphs
and $1-$factorizations that will not be explicitely defined here.
Let $G$ be a group of even order $2n$.
We use for $G$ a multiplicative notation and denote by $1_G$ its identity, we also use $1$ if the group $G$ is clear from the context.
Let us denote by $V$ and $E$ the set of vertices
and edges of $K_{2n}$, respectively.
We identify the vertices of $K_{2n}$ with the
group-elements of $G$. We shall denote by $[x,y]$ the edge with
vertices $x$ and $y$. Following \cite{Bu} we always consider
$G$ in its right regular permutation representation. In other
words, each group-element $g\in G$ is identified with the
permutation $V \to V$, $x\mapsto xg$. This
action of $G$ on $V$ induces actions on the subsets
of $V$ and on sets of such subsets. Hence if $g\in G$
is an arbitrary group-element and $S$ is any subset of
$V$ then we write $S\cdot g=\{xg\,:\ x\in S\}$. In particular,
if $S=[x,y]$ is an edge, then $[x,y]\cdot g=[xg,yg]$. Furthermore,
if $U$ is a collection of subsets of $V$, then we write
$U\cdot g=\{S\cdot g\,:\ S\in U\}$. In particular, if $U$ is a collection
of edges of $K_{2n}$ then $U\cdot g=\{[xg,yg]\,:\ [x,y]\in U\}$.
The $G$-orbit of an edge $[x,y]$ has either length $2n$ or $n$ and
we speak of a \textit{long} orbit or a \textit{short} orbit,
respectively, and we call $[x,y]$ a \textit{long}
edge or a \textit{short} edge, respectively. If
$[x,y]$ is a short edge, then there is a non-trivial group element
$g$ so that $[xg, yg] = [x,y]$. Such a $g$ is unique ($g =
x^{-1}y$) and is an involution; we call this $g$ the involution {\it
associated} with the short edge $[x,y]$.
Obviously, the element $yx^{-1}$ is an involution as well.

It is easy to show that a $1$-factor of $K_{2n}$ which
is fixed by $G$ necessarily coincides with a short $G$-orbit of
edges.

If $e$ is an edge, respectively if $S$  is a set of edges, we will denote by $Orb_G(e)$, respectively by $Orb_G(S)$,  the orbit of $e$, respectively of the set $S$, under the action of $G$.

If $H$ is a subgroup of $G$ then a system of distinct
representatives for the left cosets of $H$ in $G$ will be called a
\textit{left transversal} for $H$ in $G$.

If $[x,y]$ is an edge in
$K_{2n}$ we define
\[
\partial([x,y])=\left\{\begin{array}{lcl}
\{xy^{-1},yx^{-1}\} & & \textrm{ if }[x,y]\textrm{ is long } \\
            & &                                      \\
\{xy^{-1}\}     & & \textrm{ if }[x,y]\textrm{ is short }\\
\end{array}\right.
\]
\[
\phi([x,y])=\left\{\begin{array}{lcl}
\{x,y\} & & \textrm{ if }[x,y]\textrm{ is long } \\
        & &                                      \\
\{x\}   & & \textrm{ if }[x,y]\textrm{ is short }\\
\end{array}\right.
\]

Roughly speaking, we also say that the edge $[x,y]$ has {\it difference set} $\partial([x,y])$,
or that $\{xy^{-1},yx^{-1}\}$ are the {\it differences} of $[x,y]$.

It is clear that all the edges having a same difference set form a unique $G$-orbit.

If $S$ is a set of edges of $K_{2n}$ we define
\[
\partial S=\bigcup_{e\in S}\partial(e)
\qquad \phi(S) =\bigcup_{e\in S}\phi(e)
\]
where, in either case, the union may contain repeated elements and
so, in general, will return a multiset.

In {\rm\cite[Definition 2.1]{Bu}}
a {\it starter} in a group $G$ of even order is a set
$\Sigma=\{S_1,\dots,S_k\}$ of subsets of $E$ together with associated
subgroups $H_{1},\ldots, H_{k}$ which satisfy the following
conditions:

\begin{enumerate}[(i)]
\item $\partial S_1\cup\dots\cup\partial S_k=G\setminus\{1_G\}$;
\item for $i=1,\dots,k$, the set $\phi(S_i)$ is a left transversal
for  $H_i$ in $G$;
\item for $i=1, \dots, k$, $H_{i}$ must contain
the involutions associated with any short edge in $S_{i}$.
\end{enumerate}

We note that $G-\{1_G\}$ is a set, so that
$\partial S_1\cup\dots\cup\partial S_k$ is a list of distinct elements,
the edges of $S_1\cup \dots \cup S_k$ are all distinct and lie in distinct $G$-orbits.
Hence it also follows $S_{i}$ can
have no edges in common with $S_{j}$ for $i \neq j$.
Moreover, each $\phi(S_i)$ is a set and then the edges of $S_i$ are vertex disjoint.

It is proved in \cite{Bu}, that the existence of a starter
in a finite group $G$ of order $2n$ is equivalent to the existence of a $G-$regular $1-$factorization of
$K_{2n}$.
Property $(i)$ in previous definition ensures that every edge of $K_{2n}$ will occur in
exactly one $G$-orbit of an edge from $S_{1}\cup\ldots \cup S_{k}$. Properties $(ii)$ and $(iii)$ ensure the union of the $H_{i}$-orbits
of edges from $S_{i}$ will form a 1-factor. Namely, for each index $i$, we form a $1-$factor $F_i=\cup_{e\in S_i}Orb_{H_i}(e)$, whose
stabilizer in $G$ is the subgroup $H_i$; the $G$-orbit $Orb_G(F_i) = \{F_i^{1}, \dots , F_i^{t_i}\}$,
which has length $t_i= |G:H_i|$ (the index of $H_i$ in $G$), is then included in the $1-$factorization.

Observe also that the existence of a $1-$factor, say  $F_1$, which is fixed by $G$ is equivalent to the
existence in $\Sigma$ of a set $S_1=\{e\}$, where $e$ is a short edge.
Moreover, $\phi(S_i)$ and $\partial S_i$ both contain $t_i$ elements
and $t_i$ is equal to the number of short edges in $S_i$ plus twice the number of long edges in $S_i$.
It is also true that the unique $1-$factor which contains a chosen edge $e$ with differences in  $\partial S_i$
is one of the $1$-factors in $\{F_i^{1}, \dots , F_i^{t_i}\}$.

Suppose $n > 2$ to be even and $G$ to contain a cyclic subgroup $H$ of index $2$. Let $j$ be the unique involution
in $H$ and let $\{h_1,\dots h_{\frac{n}{2}}\}$ be a set of distinct representatives for the
cosets of $\{1,j\}$ in $H$. Suppose $\Sigma = \{S_1, \dots, S_r\}$ to be a starter in $G$
with associated subgroups $H_1, \dots ,H_r$,
and such that $S_1 = \{e\}$, with $\partial e = \{j\}$.
Let ${\cal F}$ be the $G-$regular $1-$factorization equivalent to $\Sigma$.

In the following Lemma \ref{L1} we describe  a subgraph $R$ of $K_{2n}$ which leads
to the construction of a complete set of
spanning trees orthogonal to ${\cal F}$.

\begin{lemma}\label{L1}
Let $R= R_2 \cup \dots \cup R_r$ be a subgraph of $K_{2n}$  such that:

\begin{enumerate}

\item For each $i\in \{2, \dots , r\}$, the set $R_i$ contains $t_i = [G:H_i]$ edges: one for each $1-$factor of the
set $\{F_i^{1}, \dots , F_i^{t_i}\}$, and the set of distinct elements of $\partial R_i$ coincides with $\partial S_i$.

\item If $l$ is a long edge of $R_i$, $i\in \{2, \dots , r\}$,
then there is exactly one edge $l'\in R_i$ such that
$\partial l = \partial l'$ and $l'\notin Orb_H(l)$.
While, if $l$ is a short edge of $R_i$, $i\in \{2, \dots , r\}$,
then it is the unique edge of $R_i$ with difference set $\partial l$.

\item There exist two distinct edges $e_1$ and $e_2$ of the fixed $1-$factor $F_1$ such that $Orb_H(e_1) \cap Orb_H(e_2)=\emptyset$
and both $R\cup \{e_1\}$ and $R\cup \{e_2\}$ are spanning connected graphs.

\end{enumerate}

\noindent
Let $T_1=R\cup \{e_1\}$ and $T_2=Rj\cup \{e_2\}$.

\noindent
The set ${\cal T}=\{T_1h_1, \dots , T_1h_{\frac{n}{2}}\}\cup \{T_2h_1, \dots , T_2h_{\frac{n}{2}}\}$ is a complete set of rainbow spanning trees.

\end{lemma}

\begin{proof}
Conditions 1 and 3 assures that both $T_1= R\cup \{e_1\}$ and $R\cup \{e_2\}$ are spanning connected graphs with $2n-1$ edges belonging to distinct $1-$factors, therefore they are spanning rianbow trees.  Since $(R\cup \{e_2\})j = Rj\cup \{e_2\}= T_2$, therefore $T_2$ is a rainbow spanning tree as well. We also conclude that each graph in ${\cal T}$ is a rainbow spanning tree. We now prove that ${\cal T}$ is a partition of the edge-set of $K_{2n}$.

Let $f$ be an edge of $K_{2n}$. We have three possibility: either $f$ is long, or $f$ is short with $\partial f = \{j_1\}$, $j_1\ne j$, or $f$ is short and $\partial f =\{j\}$. In all these cases we prove that $f$ belongs to a unique spanning tree of ${\cal T}$.

Suppose $f$ is a long edge, then there exists a unique $S_i \in \Sigma \setminus \{S_1\}$ such that $\partial f \in \partial S_i$
and $f$ is an edge of $F_i^1\cup \dots \cup F_i^{t_i}$. Conditions 1 and 2 assures the existence of $l,l' \in R_i$ such that
$\partial f = \partial l = \partial l'$, with $l' \notin Orb_H(l)$, $f \in Orb_G(l)= Orb_G(l')$.  Let $g_1,g_2 \in G$ be the unique elements such that $f=lg_1=l'g_2$. Since $g_1g_2^{-1}\notin H$, just one of the two elements $g_1$ or $g_2$ is in $H$ and then there is a unique graph of the set $\{Rh \ | \ h\in H\}$ containing $f$. Therefore $f$ belongs to a unique tree of ${\cal T}$.

Now suppose $f$ is short and $\partial f =\{j_1\}$, $j_1\ne j$. Let $l$ be the unique edge of $R$ with $\partial l = \partial f$.
Since $j_1\notin H$, all the $n$ edges of $K_{2n}$ with difference set $\{j_1\}$ are in $Orb_H(l)$. We conclude that a unique tree of ${\cal T}$ contains $f$.

Finally suppose $f$ to be a short edge with $\partial f = \{j\}$, i.e., $f\in F_1$.  Condition 3 implies that $F_1$ contains the $n$ distinct
edges $\{e_1h_i, e_2h_i \ | \ i=1, \dots , \frac{n}{2}\}$. Therefore,  a unique tree of ${\cal T}$ contains $f$.
\end{proof}

\section{Dicyclic groups and complete sets of rainbow spanning trees}\label{dic}

In this section we prove the following Proposition \ref{Dic}

\begin{proposition}\label{Dic}
Let $G$ be a dicyclic group of order $2n \ge 6 $. There exists a $G-$regular $1-$factorization of $K_{2n}$ together with a complete set of rainbow spanning trees.
\end{proposition}

The dicyclic group $G$ of order $2n=4s$, $s\ge 2$, can be presented as  follows \cite[p.189]{Sc}:
\[
G=\langle a,b \,: \ a^{2s}=1, b^2=a^s, b^{-1}ab=a^{-1}\rangle.
\]
We have $G=\{1,a,\dots,a^{2s-1},b,ba,\dots,ba^{2s-1}\}$
and the relations $a^rb=ba^{-r}$, $ba^r(ba^t)^{-1}=a^{t-r}$, $(ba^r)^{-1}=ba^{r+s}$, $(ba^r)^2=a^s$ hold
for $r$, $t=0,1,\dots,(2s-1)$. Furthermore $a^s$ is the
unique involution in $G$. In particular, if $s=2^{m-1}$,
then $G$ is a generalized quaternion group of order
$2^{m+1}$.

We consider the $G-$regular $1-$factorizaion of $K_{2n}$ constructed in \cite{BR}. The description is given in terms of starters
according to whether $s$ is even or odd.

\vskip0.3truecm\noindent
{\bf Starter in the case $s$ even}

\noindent
A starter can be constructed as follows:

\vskip0.3truecm\noindent
$\Sigma =\{S\} \cup \{S_{2i+1} \ | \ 0\le i \le \frac{s-2}{2}\} \cup \{S^*_{j} \ | \ 0\le j \le s-1, j\ne \dfrac{s}{2}\} \cup \{S_s\}$;
With:
\vskip0.3truecm\noindent
$S=\{[a^t,a^{-t}], \ t=1, \dots \frac{s}{2}-1\}\cup \{[1,ba^{\frac{s}{2}}]\}$;

\vskip0.3truecm\noindent
$S_{2i+1}=\{[1,a^{2i+1}]\}$, \ \ $0\le i \le \frac{s-2}{2}$;

\vskip0.3truecm\noindent
$S^*_{j}=\{[1,ba^j]\}$, \ \ $0\le j \le s-1, j\ne \dfrac{s}{2}$;

\vskip0.3truecm\noindent
$S_s=\{[1,a^s]\}$.

\vskip0.3truecm\noindent
Take the subgroups:

\vskip0.3truecm\noindent
$<b>=\{1,b,a^s,ba^s\}$  and $<b,a^2> = \{1, a^2,a^4, \dots, a^{2n-2}, b, ba^2, \dots, ba^{2n-2}\}$.

\vskip0.3truecm\noindent
We have:

\vskip0.3truecm\noindent
$\partial S = \{a^{2t}, a^{-2t} \ | \ t=1, \dots, \frac{s}{2}-1\}\cup \{ba^{\frac{s}{2}}, ba^{-\frac{s}{2}}\}$
and $\phi(S)$ is a left transversal for $<b>$.

\vskip0.3truecm\noindent
$\partial S_{2i+1}=\{a^{2i+1},a^{-2i-1}\}$ and $\phi(S_{2i+1}) =\{1, a^{2i+1}\}$ is a left transversal for the subgroup
$<b,a^2>$.

\vskip0.3truecm\noindent
$\partial S^*_{j} = \{ba^j, ba^{j+s}\}$  and $\phi(S^*_j)=\{1, ba^j\}$ is a left transversal for the
cyclic subgroup $<a>$.

\vskip0.3truecm\noindent
$\partial S_s = \{a^s\}$ and $\phi(S_s)=\{1\}$.

\vskip0.3truecm\noindent
With the starter above, we construct the following $1-$factors:

\vskip0.3truecm\noindent
$F=Orb_{<b>}(S)= \{[1,ba^{\frac{s}{2}}], [b,a^{\frac{s}{2}}], [a^s,ba^{s+\frac{s}{2}}], [ba^s,a^{s+\frac{s}{2}}], [a^t,a^{-t}],
[ba^{-t},ba^{t}]$,
$[a^{s+t},a^{s-t}], [ba^{s-t},ba^{s+t}], t=1,\dots , \frac{s}{2}-1 \}$.

\vskip0.3truecm\noindent
$F_{2i+1}= Orb_{<b,a^2>}(S_{2i+1})= \{[a^{2k},a^{2i+1+2k}], [ba^{2k},ba^{2k-2i-1}],  \ k=0, \dots, s-1\}$ with $0\le i < \frac{s-2}{2}$.

\vskip0.3truecm\noindent
$F^*_j=Orb_{<a>}(S^*_j)=\{[a^k, ba^{j+k}], \ k=0, \dots, 2s-1\}$ with $0\le j \le s-1$, \ $j\ne \frac{s}{2}$.

\vskip0.3truecm\noindent
$F_s= Orb_G([1,a^s])$.

These $1-$factors give rise to the $1-$factorization. Namely:

\noindent
The $1-$factor $F$ is fixed by $<b>$ and its orbit under $G$ yields the $1-$factors:

$$F, Fa, Fa^2, \dots, Fa^{s-1}$$

\noindent
These $1-$factors cover all long edges with difference set in $\partial S$.

For each $0\le i \le \frac{s-2}{2}$, the $1-$factor $F_{2i+1}$ is fixed by $<b,a^2>$ and its orbit under $G$ yields the $1-$factors:

$$F_{2i+1}, F_{2i+1}a$$

\noindent
These $1-$factors cover all edges with difference set in $\partial S_{2i+1}$.

For each $0\le j \le s-1$, \ $j\ne \frac{s}{2}$, the $1-$factor $F^*_j$ is fixed by $<a>$ and its orbit under $G$ yields the $1-$factors:

$$F^*_j, F^*_jb$$

\noindent
These $1-$factors cover all edges with difference set in $\partial S^*_j$.

\noindent
Finally $F_s$ is a fixed $1-$factor which contains all edges with difference set $\{a^s\}$.

\vskip0.3truecm\noindent
{\bf Starter in the case $s$ odd}

\noindent
A starter can be constructed as follows:

\vskip0.3truecm\noindent
$\Sigma =\{S\} \cup \{S^*_{i} \ | \ 0\le i \le s-1, i \ne \frac{s-1}{2}\} \cup \{S_s\}$.
With:
\vskip0.3truecm\noindent
$S=\{[a^t,a^{s-t-1}], [ba^t,ba^{s-t-2}] \ | \ 0\le t \le \frac{s-3}{2}\}\cup \{[a^{\frac{s-1}{2}},ba^{s-1}]\}$;

\vskip0.3truecm\noindent
$S^*_i=\{[1,ba^i]\}$, $0\le i \le s-1$ \ $i\ne \frac{s-1}{2};$

\vskip0.3truecm\noindent
$S_s=\{[1,a^s]\}$

\vskip0.3truecm\noindent
We have:
\vskip0.3truecm\noindent
$\partial S = \{a^j,  \ 1\le j \le 2s-1, j\ne s\}\cup \{ba^{\frac{s-1}{2}}, ba^{s+\frac{s-1}{2}}\}$
and $\phi(S)$ is a left transversal for $<a^s>$.

\vskip0.3truecm\noindent
$\partial S^*_i=\{ba^i, ba^{i+s}\}$ and $\phi(S^*_i)$ is a left transversal for the subgroup
$<a>$.

\vskip0.3truecm\noindent
$\partial S_s = \{a^s\}$ and $\phi(S_s)=\{1\}$.

\vskip0.3truecm\noindent
With the starter above, we construct the following $1-$factors:

\vskip0.3truecm\noindent
$F=Orb_{<a^s>}(S)= \{[a^t,a^{s-t-1}], [a^{t+s},a^{2s-t-1}], [ba^t,ba^{s-t-2}], [ba^{t+s},ba^{2s-t-2}],$

$[a^{\frac{s-1}{2}},ba^{s-1}], [a^{s+\frac{s-1}{2}},ba^{2s-1}], \ | \ 0\le t \le \frac{s-3}{2}\}$

\vskip0.3truecm\noindent
$F^*_i= Orb_{<a>}(S^*_i)= \{[a^r,ba^{i+r}],  \ r=0, \dots, 2s-1\}$ with $0\le i < s-1$, $i\ne \frac{s-1}{2}$.

\vskip0.3truecm\noindent
$F_s= Orb_G([1,a^s])$.

These $1-$factors give rise to the $1-$factorization. Namely:

\noindent
The $1-$factor $F$ is fixed by $<a^s>$ and its orbit under $G$ yields the $1-$factors:

$$F, Fa, Fa^2, \dots, Fa^{s-1}, Fb, Fba, Fba^2, \dots, Fba^{s-1}$$

\noindent
These $1-$factors cover all long edges with difference set in $\partial S$.

For each $0\le i \le s-1$, $i\ne \frac{s-1}{2}$, the $1-$factor $F^*_i$ is fixed by $<a>$ and its orbit under $G$ yields the $1-$factors:

$$F^*_i, F^*_ib$$

\noindent
These $1-$factors cover all edges with difference set in $\partial S^*_i$.

\noindent
Finally $F_s$ is a fixed $1-$factor which contains all edges with difference set $\{a^s\}$.

\vskip 0.3truecm\noindent
We are now able to construct a complete set of rainbow spanning trees in both of these two cases using the method explained in the previous Lemma \ref{L1}.

\noindent
\subsection{Case $s$ even}\label{diceven}

Let $s\equiv 2$ (mod $4$).

\noindent
Suppose $s\ge 6$. Consider the forest induced by the following set $T$ of edges:

$$T=\{[1,ba^{\frac{s}{2}}], [1,ba^{s+\frac{s}{2}}], [1,a^{2t}], [b,ba^{s+2t}] \ | \ t=1,\dots , \dfrac{s}{2}-1 \}.$$

We have $[1,ba^{\frac{s}{2}}]\in F$, $[1,ba^{s+\frac{s}{2}}]\in Fa^{\frac{s}{2}}$, we also have: $[1,a^{2t}]\in Fa^t$, $[b,ba^{s+2t}]\in Fa^{\frac{s}{2}+t}$. In fact: $[1,ba^{s+\frac{s}{2}}]=[a^{s+\frac{s}{2}},ba^s]a^{\frac{s}{2}}$, $[1,a^{2t}]= [a^{-t},a^t]a^t$,
$[b,ba^{s+2t}]=[ba^{s+(\frac{s}{2}-t)}, ba^{s-(\frac{s}{2}-t)}]a^{\frac{s}{2}+t}$.  Moreover $\partial [1,a^{2t}] = \partial [b,ba^{s+2t}]$ and these two long edges are in distinct orbits under the action of $<a>$ for each $t=1,\dots,\frac{s}{2}-1$, and also $\partial [1,ba^{\frac{s}{2}}] = \partial [1,ba^{s+\frac{s}{2}}]$ and these two long edges are in distinct orbits under $<a>$.

For each $i= 0, \dots, \frac{s-2}{2}$, let $T_{2i+1} = \{[1,a^{2i+1}], [b,ba^{2i+1}]\}$. We have $[1,a^{2i+1}]\in F_{2i+1}$ and $[b,ba^{2i+1}]\in F_{2i+1}a$, in fact $[b,ba^{-2i-1}]\in F_{2i+1}$ and then $[ba^{2i+1},b]\in F_{2i+1}a^{2i+1} =F_{2i+1}a$ since $F_{2i+1}a^2=F_{2i+1}$. Moreover $\partial [1,a^{2i+1}] = \partial [b,ba^{2i+1}]$ and these two long edges are in distinct orbits under $<a>$.

Set $T'=T\cup (\bigcup_{i=0}^{\frac{s-2}{2}}T_{2i+1})$. The graph $T'$ is a rainbow tree which is given by the union of a star at $1$ and a star at $b$ which are connected through the edge $[1,ba^{\frac{s}{2}}]$. Moreover $T'$ covers all the vertices of $K_{4s}$ except for those in the set:
$$\{a^{s+i} \ | \ 0 \le i \le s-1\}\cup \{ba^{s-2t} \ | \ 0\le t \le \frac{s}{2}-1\}\cup$$

$$ \cup \{ba^{s+2j+1} \ | \ 0\le j \le \frac{s-2}{2}, j\ne \frac{s-2}{4}\}$$

Consider the  star at $a^s$ induced by the set:
$$S_1=\{[a^s,ba^{s-2t}], [a^s,ba^{s+2j+1}] \ | \ 0\le t \le \frac{s-2}{2}, 0\le j \le \frac{s-2}{2}, j\ne \frac{s-2}{4} \},$$
together with the star at $ba^{s+1}$ induced by:
$$S_2=\{[ba^{s+1},a^{2s-2j}] \ | \ 1\le j \le \frac{s-2}{2}, j\ne \frac{s-2}{4}\},$$
and the star at $ba^{2s-1}$ induced by the set:
$$S_3=\{[ba^{2s-1}, a^{s+2t-1}] \ | \ 1\le t \le \frac{s-2}{2}\}$$.

Now let:
$$T''=S_1\cup S_2 \cup S_3 \cup \{[ba^s,a^{2s-1}],[ba^{\frac{s}{2}+1}, a^{s+\frac{s}{2}+1}]\}.$$

The graph $T''$ is a tree, $T'$ and $T''$ are disconnected and all together cover all the vertices of $K_{4s}$.
Moreover, you can partition $T''$ into the following pairs of edges:

\vskip 0.3truecm\noindent
$T^*_0=\{[a^s,ba^s], [ba^{\frac{s}{2}+1},a^{s+\frac{s}{2}+1}]\}$, $T^*_1=\{[a^s,ba^{s+1}], [ba^s,a^{2s-1}]\}$,

\vskip 0.3truecm\noindent
$T^*_{2j+1}=\{[a^s,ba^{s+2j+1}],[ba^{s+1},a^{2s-2j}]\}$ with $1\le j \le \frac{s-2}{2}, j\ne \frac{s-2}{4}$,

\vskip 0.3truecm\noindent
$T^*_{s-2t}=\{[ba^{2s-1}, a^{s+2t-1}], [a^s,ba^{s-2t}]\}$ with $1\le t \le \frac{s-2}{2}$.

\vskip 0.3truecm
Observe that the edges of $T^*_0$ have the same difference set $\{b, ba^s\}$, are in distinct orbits under $<a>$ and they belong  to $F^*_0$ and $F^*_0b$, respectively. In fact:  $[1,b]\in F^*_0$, $F^*_0$ is fixed by $<a>$ and then: $[a^s,ba^s]\in F^*_0$, moreover $[ba^{\frac{s}{2}+1},a^{s+\frac{s}{2}+1}]=[ba^{\frac{s}{2}+1},b^2a^{\frac{s}{2}+1}]=[a^{-\frac{s}{2}-1}b,ba^{-\frac{s}{2}-1}b]=
[a^{-\frac{s}{2}-1},ba^{-\frac{s}{2}-1}]b \in F^*_0b$.

\vskip 0.3truecm
Observe that the edges of $T^*_1$ have the same difference set $\{ba, ba^{s+1}\}$, are in distinct orbits under $<a>$ and they belong  to $F^*_1$ and $F^*_1b$, respectively. In fact: $[1,ba]\in F^*_1$, $F^*_1$ is fixed by $<a>$ and then: $[a^s,ba^{s+1}]\in F^*_1$, moreover:
$[a^sb,ba^{s+1}b]\in F^*_1b$ and $[a^sb,ba^{s+1}b]=[ba^s,b^2a^{-s-1}]=[ba^s,a^{2s-1}]$.

\vskip 0.3truecm
For each $1\le j \le \frac{s-2}{2}, j\ne \frac{s-2}{4}$, the edges of $T^*_{2j+1}$ have the same difference set $\{ba^{2j+1}, ba^{s+2j+1}\}$, are in distinct orbits under $<a>$ and they belong  to $F^*_{2j+1}$ and $F^*_{2j+1}b$, respectively.
In fact: $[1,ba^{2j+1}]\in F^*_{2j+1}$, $F^*_{2j+1}$ is fixed by $<a>$ and then: $[a^s,ba^{s+2j+1}]\in F^*_{2j+1}$, moreover:

\noindent
$[a^{-s-1},ba^{-s-1+2j+1}]b\in F^*_{2j+1}b$ and $[a^{-s-1}b,ba^{-s-1+2j+1}b]=[ba^{s+1},b^2a^{s-2j}]=[ba^{s+1},a^{2s-2j}]$.

\vskip 0.3truecm
When $1\le t \le \frac{s}{2}-1$, the edges of $T^*_{s-2t}$ have the same difference set

\noindent
$\{ba^{s-2t}, ba^{2s-2t}\}$, are in distinct orbits under $<a>$ and they belong  to $F^*_{s-2t}$ and $F^*_{s-2t}b$, respectively.
In fact: $[1,ba^{s-2t}]\in F^*_{s-2t}$, $F^*_{s-2t}$ is fixed by $<a>$ and then: $[1,ba^{s-2t}]a^{s+2t-1}=[a^{s+2t-1},ba^{2s-1}]\in F^*_{s-2t}$, moreover:
$[1,ba^{s-2t}]b\in F^*_{s-2t}b$ and $[b,ba^{s-2t}b]=[b,b^2a^{-s+2t}]=[b,a^{2t}]$. Therefore $[b,a^{2t}]\in F^*_{s-2t}b$  and
$[ba^{s-2t},a^s]\in F^*_{s-2t}ba^{s-2t}$ with $F^*_{s-2t}ba^{s-2t}=F^*_{s-2t}a^{-s+2t}b=F^*_{s-2t}b$.

Therefore, the graph $R=T' \cup T''$ satisfies conditions (1) and (2) of Lemma \ref{L1}. Let now $e_1=[1,a^s]\in F_s$ and $e_2=[b,ba^s]\in F_s$, they are in distinct orbits under $<a>$ and both connect $T'$ and $T''$ in such a way that $R\cup \{e_1\}$ and $R\cup \{e_2\}$ satisfy condition (3) of Lemma \ref{L1}. We conclude that ${\cal T}=\{T_1a^i \ | \ 0\le i \le s-1\}\cup \{T_2a^i \ | \ 0\le i \le s-1\}$ with $T_1=R\cup \{e_1\}$ and $T_2=Ra^s\cup \{e_2\}$ is a complete set of rainbow spanning trees.

If $s=2$, the dicyclic group is the quaternion group $Q_8$ and it is easy to observe that $R=T'\cup T"$ with $T'=\{[1,ba],[1,ba^3],[1,a],[b,ba],[ba,a^3]\}$ and $T"=\{[a^2,ba^2\}$ is rainbow and satisfies (1) and (2) of Lemma \ref{L1} and the above construction can be repeated with $e_1 =[1,a^2]$ and $e_2=[b,ba^2]$.

For the readers' convenience, in the following Figures 1 we picture $R\cup \{e_1\}$ and $Ra^s \cup \{e_2\}$ when $s=2$ and we point out $e_1$ and $e_2$ with a different color. In the followig Figure 2 we show $R\cup \{e_1\}$ when $s=6$, in particular
we picture the sets $T'$, $T''$ and the edge $e_1$ assigning a color to each of them.

\centerline{
\begin{tikzpicture}
  [scale=.2]
  \node (N-1) at (-2.5,5) {$R\cup \{e_1\}$};
  \node (N0) at (27,5) {$Ra^2 \cup \{e_2\}$};
  \node (N0.1) at (25,-10) {Figure 1: \ case $s=2$. Group $Q_8$.};
  \node [circle, draw](n1) at (1,10) {};
  \node (N1) at (-0.7,10) {$1$};
  \node [circle, draw](n2) at (6,13)  {};
  \node (N2) at (6,14.8) {$a$};
  \node [circle, draw](n3) at (11,13)  {};
  \node (N3) at (11,15.3) {$a^2$};
  \node [circle, draw](n4) at (16,13) {};
   \node (N4) at (16,15.3) {$a^3$};
  \node [circle, draw](n5) at (30,10)  {};
   \node (N5) at (28.5,11) {$a^2$};
  \node [circle, draw](n6) at (35,13)  {};
   \node (N6) at (35,15.3) {$a^3$};
  \node [circle, draw](n7) at (40,13) {};
   \node (N7) at (40,15.3) {$1$};
  \node [circle, draw](n8) at (45,13)  {};
   \node (N8) at (45,15.3) {$a$};
  \node [circle, draw](n9) at (1,0)  {};
   \node (N9) at (-0.7,0) {$b$};
   \node [circle, draw](n10) at (6,-3) {};
   \node (N10) at (6,-5.3) {$ba$};
  \node [circle, draw](n11) at (11,-1.5)  {};
  \node (N11) at (11,-3.3) {$ba^2$};
  \node [circle, draw](n12) at (16,-3)  {};
  \node (N12) at (16,-5.3) {$ba^3$};
  \node [circle, draw](n13) at (30,0) {};
  \node (N13) at (27,0) {$ba^2$};
  \node [circle, draw](n14) at (35,-3)  {};
  \node (N14) at (35,-5.3) {$ba^3$};
  \node [circle, draw](n15) at (40,-1.5)  {};
  \node (N15) at (42,-2) {$b$};
  \node [circle, draw](n16) at (45,-3) {};
  \node (N16) at (45,-5.3) {$ba$};

\draw (n1) -- (n2);
\draw (n1) -- (n12);
\draw (n1) -- (n10);
\draw (n9) -- (n10);
\draw (n10) -- (n4);
\draw (n3) -- (n11);
\draw [blue](n1) -- (n3);
\draw (n5) -- (n6);
\draw (n5) -- (n14);
\draw (n5) -- (n16);
\draw (n13) -- (n14);
\draw (n14) -- (n8);
\draw (n7) -- (n15);
\draw [red](n13) -- (n15);
\end{tikzpicture}
}

\vskip 0.5truecm
\centerline{
\begin{tikzpicture}
  [scale=.2]
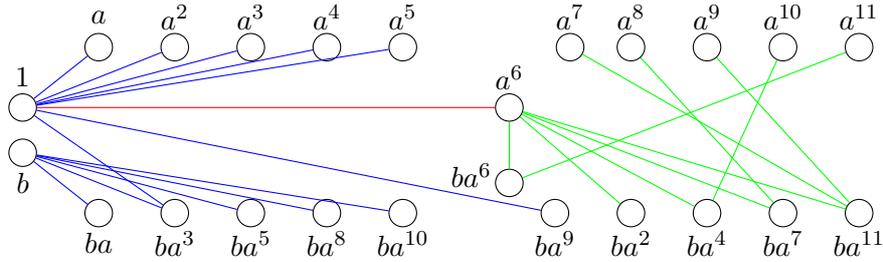
  \node (N0.1) at (25,-3) {Figure 2: $R\cup \{e_1\}$, case $s=6$. Dicyclic group of order 24.};
  \node [circle, draw](n1) at (-2,10) {};
  \node (N1) at (-2,12) {$1$};
  \node [circle, draw](n2) at (3,14)  {};
  \node (N2) at (3,16) {$a$};
  \node [circle, draw](n3) at (8,14)  {};
  \node (N3) at (8,16) {$a^2$};
  \node [circle, draw](n4) at (13,14)  {};
  \node (N4) at (13,16) {$a^3$};
  \node [circle, draw](n5) at (18,14) {};
   \node (N5) at (18,16) {$a^4$};
  \node [circle, draw](n6) at (23,14)  {};
   \node (N6) at (23,16) {$a^5$};
  \node [circle, draw](n7) at (30,5)  {};
   \node (N7) at (27.5,5.5) {$ba^6$};
  \node [circle, draw](n8) at (34,14) {};
   \node (N8) at (34,16) {$a^7$};
  \node [circle, draw](n9) at (38,14)  {};
   \node (N9) at (38,16) {$a^8$};
   \node [circle, draw](n10) at (43,14)  {};
   \node (N10) at (43,16) {$a^9$};
   \node [circle, draw](n11) at (48,14)  {};
   \node (N11) at (48,16) {$a^{10}$};
   \node [circle, draw](n12) at (53,14)  {};
   \node (N12) at (53,16) {$a^{11}$};
  \node [circle, draw](n13) at (-2,7)  {};
   \node (N13) at (-2,5) {$b$};
   \node [circle, draw](n14) at (3,3) {};
   \node (N14) at (3,1) {$ba$};
  \node [circle, draw](n15) at (8,3)  {};
  \node (N15) at (8,1) {$ba^3$};
  \node [circle, draw](n16) at (13,3)  {};
  \node (N16) at (13,1) {$ba^5$};
  \node [circle, draw](n17) at (18,3) {};
  \node (N17) at (18,1) {$ba^8$};
  \node [circle, draw](n18) at (23,3)  {};
  \node (N18) at (23,1) {$ba^{10}$};
  \node [circle, draw](n19) at (30,10)  {};
  \node (N19) at (30,12) {$a^6$};
  \node [circle, draw](n20) at (33,3) {};
  \node (N20) at (33,1) {$ba^9$};
  \node [circle, draw](n21) at (38,3) {};
  \node (N21) at (38,1) {$ba^2$};
  \node [circle, draw](n22) at (43,3) {};
  \node (N22) at (43,1) {$ba^4$};
  \node [circle, draw](n23) at (48,3) {};
  \node (N23) at (48,1) {$ba^7$};
  \node [circle, draw](n24) at (53,3) {};
  \node (N24) at (53,1) {$ba^{11}$};

\draw [blue](n1) -- (n2);
\draw [blue](n1) -- (n3);
\draw [blue](n1) -- (n4);
\draw [blue](n1) -- (n5);
\draw [blue](n1) -- (n6);
\draw [blue](n1) -- (n15);
\draw [blue](n1) -- (n20);
\draw [blue](n13) -- (n14);
\draw [blue](n13) -- (n15);
\draw [blue](n13) -- (n16);
\draw [blue](n13)--(n17);
\draw [blue](n13)--(n18);
\draw [green](n7) -- (n19);
\draw [green](n7) -- (n12);
\draw [green](n19) -- (n21);
\draw [green](n19) -- (n22);
\draw [green](n19) -- (n23);
\draw [green](n19)--(n24);
\draw [green](n8) -- (n24);
\draw [green](n9)--(n23);
\draw [green](n11) -- (n22);
\draw [green](n10)--(n24);
\draw [red](n1) -- (n19);
\end{tikzpicture}
}

\vskip0.5truecm\noindent
Let $s\equiv 0$ (mod $4$).

\noindent
With a slightly modification of the construction above, we construct a complete set of rainbow spanning trees. Namely, take  $T'=T\cup (\bigcup_{i=0}^{\frac{s-2}{2}}T_{2i+1})$ exactly as above and recall that $T'$ is a rainbow tree. It is the union of a star at $1$ together with a star at $b$ connected through the edge $[1,ba^{s+\frac{s}{2}}]$.
Let
$$S_1=\{[a^s,ba^{s-2t}], [a^s,ba^{s+2j+1}] \ | \ 0\le t \le \frac{s-2}{2}, t\ne \frac{s}{4},  0\le j \le \frac{s-2}{2}\},$$

$$S_2=\{[ba^{s+1},a^{2s-2j}] \ | \ 1\le j \le \frac{s-2}{2}\},$$

$$S_3=\{[ba^{2s-1}, a^{s+2t-1}] \ | \ 1\le t \le \frac{s-2}{2} \ | \ t\ne \frac{s}{4}\},$$

$$T''=S_1\cup S_2 \cup S_3 \cup \{[ba^s,a^{2s-1}]\}.$$

It is easy to observe that $T''$ and $T'\cup \{[ba^{\frac{s}{2}-1}, a^{s+\frac{s}{2}-1}\}$ are trees, they are disconnected and all together cover all the vertices of $K_{4s}$.
Moreover, you can partition $T''\cup \{[ba^{\frac{s}{2}-1}, a^{s+\frac{s}{2}-1}\}$ into the following pairs of edges:
\vskip 0.3truecm\noindent
$T^*_0=\{[a^s,ba^s], [ba^{\frac{s}{2}-1},a^{s+\frac{s}{2}-1}]\}$, $T^*_1=\{[a^s,ba^{s+1}], [ba^s,a^{2s-1}]\}$,

\vskip 0.3truecm\noindent
$T^*_{2j+1}=\{[a^s,ba^{s+2j+1}],[ba^{s+1},a^{2s-2j}]\}$ with $1\le j \le \frac{s-2}{2}$,

\vskip 0.3truecm\noindent
$T^*_{s-2t}=\{[ba^{2s-1}, a^{s+2t-1}], [a^s,ba^{s-2t}]\}$ with $1\le t \le \frac{s-2}{2}, \ | \ t\ne \frac{s}{4}$.

Proceeding as above,  we can conclude that $R=T' \cup \{[ba^{\frac{s}{2}-1}, a^{s+\frac{s}{2}-1}\} \cup T''$ satisfies conditions (1) and (2) of Lemma \ref{L1}. Taking $e_1=[1,a^s]\in F_s$ and $e_2=[b,ba^s]\in F_s$ the set ${\cal T}=\{T_1a^i \ | \ 0\le i \le s-1\}\cup \{T_2a^i \ | \ 0\le i \le s-1\}$, with $T_1=R\cup \{e_1\}$ and $T_2=Ra^s\cup \{e_2\}$, is a complete set of rainbow spanning trees.

In the followig Figure 3 we show $R\cup \{e_1\}$ when $s=4$, in particular
we picture the sets $T'\cup \{[ba^{\frac{s}{2}-1}, a^{s+\frac{s}{2}-1}]\}$, $T''$ and the edge $e_1$ assigning a color to each of them.

\vskip 0.5truecm
\centerline{
\begin{tikzpicture}
  [scale=.2]
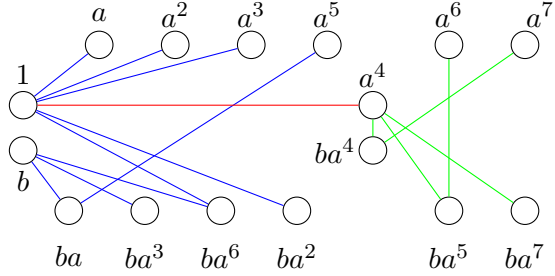
  \node (N0.1) at (40,-3) {\ \ \ \ \ \ Figure 3: $R\cup \{e_1\}$, case $s=4$. Dicyclic group of order 16.};
  \node [circle, draw](n1) at (30,10) {};
  \node (N1) at (30,12) {$1$};
  \node [circle, draw](n2) at (35,14)  {};
  \node (N2) at (35,16) {$a$};
  \node [circle, draw](n3) at (40,14)  {};
  \node (N3) at (40,16) {$a^2$};
  \node [circle, draw](n4) at (45,14)  {};
  \node (N4) at (45,16) {$a^3$};
  \node [circle, draw](n5) at (50,14) {};
   \node (N5) at (50,16) {$a^5$};
  \node [circle, draw](n6) at (53,10)  {};
   \node (N6) at (53,12) {$a^4$};
  \node [circle, draw](n7) at (58,14)  {};
   \node (N7) at (58,16) {$a^6$};
  \node [circle, draw](n8) at (63,14) {};
   \node (N8) at (64,16) {$a^7$};
  \node [circle, draw](n9) at (30,7)  {};
   \node (N9) at (30,5) {$b$};
   \node [circle, draw](n10) at (33,3) {};
   \node (N10) at (33,0) {$ba$};
  \node [circle, draw](n11) at (38,3)  {};
  \node (N11) at (38,0) {$ba^3$};
  \node [circle, draw](n12) at (43,3)  {};
  \node (N12) at (43,0) {$ba^6$};
  \node [circle, draw](n13) at (48,3) {};
  \node (N13) at (48,0) {$ba^2$};
  \node [circle, draw](n14) at (53,7)  {};
  \node (N14) at (50.5,7) {$ba^4$};
  \node [circle, draw](n15) at (58,3)  {};
  \node (N15) at (58,0) {$ba^5$};
  \node [circle, draw](n16) at (63,3) {};
  \node (N16) at (63,0) {$ba^7$};

\draw [blue](n1) -- (n2);
\draw [blue](n1) -- (n3);
\draw [blue](n1) -- (n4);
\draw [blue](n1) -- (n13);
\draw [blue](n1) -- (n12);
\draw [blue](n9) -- (n10);
\draw [blue](n9) -- (n11);
\draw [blue](n9) -- (n12);
\draw [blue](n10)--(n5);
\draw [green](n6) -- (n14);
\draw [green](n6) -- (n15);
\draw [green](n6) -- (n16);
\draw [green](n14) -- (n8);
\draw [green](n7) -- (n15);
\draw [red](n1) -- (n6);
\end{tikzpicture}
}

\noindent
\subsection{Case $s$ odd}\label{dicodd}

Consider the forest $T'$ induced by the follwing set of edges:

$$\{[1,a^{2t}], [b,ba^{2s-2t}], [1,a^{2t-1}], [b, ba^{2s-2t+1}] \ | \ 1\le t \le \frac{s-1}{2}\}\cup$$

$$\cup  \{[a^{\frac{s+1}{2}}, ba^s], [a^s,ba^{\frac{s-1}{2}}]\}.$$

\vskip 0.3truecm\noindent
Observe that $T'$ is rainbow as it contains exactly one edge for each $1-$factor of the set $\{Fa^i, Fba^i \ | \ 1\le i \le s-1\}$.

More precisely: $[1,a^{2t}]\in Fa^{\frac{s+1}{2}+t}$, $[b,ba^{2s-2t}]\in Fba^{\frac{s-1}{2}-t}$,   $[1,a^{2t-1}]\in Fba^{\frac{s-3}{2}+t}$,
$[b, ba^{2s-2t+1}]\in Fa^{\frac{s+3}{2}-t}$,  $1\le t \le \frac{s-1}{2}$, and also $[a^{\frac{s+1}{2}}, ba^s]\in Fa$,
$[a^s,ba^{\frac{s-1}{2}}]\in Fba^{s-1}$.

\noindent

In fact: $[1,a^{2t}]=[a^{\frac{s-1}{2}-t}, a^{s-\frac{s-1}{2}+t-1}]a^{-\frac{s-1}{2}+t}\in Fa^{-\frac{s-1}{2}+t}=Fa^{\frac{s+1}{2}+t}$,

\noindent
$[b,ba^{2s-2t}]=[1,a^{2t}]b\in Fa^{\frac{s+1}{2}+t}b= Fa^{s-\frac{s-1}{2}+t}b= Fba^{\frac{s-1}{2}-t}$,

\noindent
$[1,a^{2t-1}]=[ba^{\frac{s-3}{2}+t},ba^{\frac{s-1}{2}-t}]ba^{s+\frac{s-3}{2}+t}\in Fba^{s+\frac{s-3}{2}+t}=Fba^{\frac{s-3}{2}+t}$,

\noindent
$[b,ba^{2s-2t+1}]=[1,a^{2t-1}]b\in Fba^{\frac{s-3}{2}+t}b= Fa^{s-\frac{s-3}{2}-t}= Fa^{\frac{s+3}{2}-t}$,

\noindent
$[a^{\frac{s+1}{2}},ba^s]= [a^{\frac{s-1}{2}},ba^{s-1}]a\in Fa$.

\noindent
Moreover, you can partition $T'$ into the following pairs of edges:

\noindent
$\{[1,a^{2t}], [b,ba^{2s-2t}]\}$, $\{[1,a^{2t-1}], [b, ba^{2s-2t+1}]\}$, $1\le t \le \frac{s-1}{2}$ and

\noindent
$\{[a^{\frac{s+1}{2}}, ba^s],[a^s,ba^{\frac{s-1}{2}}]\}$. Two edges in the same pair are in distinct orbits under $<a>$ and have the same difference set. Namely:
$\partial [1,a^{2t}]= \partial [b,ba^{2s-2t}]= \{a^{2t}, a^{2s-2t}\}$, $\partial [1,a^{2t-1}] = \partial [b, ba^{2s-2t+1}]= \{a^{2t-1}, a^{2s-2t+1}\}$, $1\le t \le \frac{s-1}{2}$, and
$\partial [a^{\frac{s+1}{2}}, ba^s] = \partial [a^s,ba^{\frac{s-1}{2}}]= \{ba^{s+\frac{s-1}{2}},ba^{\frac{s-1}{2}}\}$.

\noindent
Consider the forest $T''$ induced by the following set of edges:

\vskip0.3truecm\noindent
$$\{[1,ba^i], [ba^s,a^{2s-i}] \ | \ 1\le i \le s-1, \ i\ne \frac{s-1}{2}\} \cup$$

$$\cup \{[a^{s+\frac{s-1}{2}}, ba^{\frac{s-1}{2}}],[a^{s+\frac{s+1}{2}}, ba^{s+\frac{s+1}{2}}]\}$$.

\vskip0.3truecm\noindent
Observe that: $[a^{s+\frac{s+1}{2}}, ba^{s+\frac{s+1}{2}}]\in F^*_0$ and $[a^{s+\frac{s-1}{2}}, ba^{\frac{s-1}{2}}]\in F^*_0b$.
In fact: the first edge is contained in $Orb_{<a>}([1,b])$, while $[a^{s+\frac{s-1}{2}}, ba^{\frac{s-1}{2}}]$ is contained in
$Orb_{<a>}([a^s,b])$ with $[a^s,b] = [1,b]b\in F^*_0b$. Moreover,  these two edges have the same difference set and are in distinct orbits under $<a>$.

\noindent
For each $i$, with $1\le i \le s-1, \ i\ne \frac{s-1}{2}$,  we obviously have  $[1,ba^i]\in F^*_i$ and $[ba^s,a^{2s-i}]\in F^*_ib$ and both these edges have the same difference set and are in distinct orbits under $<a>$.

\noindent
 The graph $T' \cup T''$ covers all the vertices of $K_{4s}$ and it is formed by two connected components.
 Namely: a first component is given by a star at $1$ connected to a star at $ba^s$ through the edge $[a^{\frac{s+1}{2}}, ba^s]$, plus the two edges
 $[ba^{\frac{s-1}{2}}, a^s]$, $[ba^{\frac{s-1}{2}}, a^{s+\frac{s-1}{2}}]$. A second component is given by a star at $b$ plus the edge $[ba^{s+\frac{s+1}{2}}, a^{s+\frac{s+1}{2}}]$.

\noindent
Moreover $R=T' \cup T''$ satisfies conditions (1) and (2) of Lemma \ref{L1}.

\noindent
Taking $e_1=[a^{\frac{s+1}{2}},a^{s+\frac{s+1}{2}}]\in F_s$ and $e_2=[ba^{\frac{s+1}{2}},ba^{s+\frac{s+1}{2}}]\in F_s$ the set ${\cal T}=\{T_1a^i \ | \ 0\le i \le s-1\}\cup \{T_2a^i \ | \ 0\le i \le s-1\}$, with $T_1=R\cup \{e_1\}$ and $T_2=Ra^s\cup \{e_2\}$, is a complete set of rainbow spanning trees.

In the followig Figure 4 we show $R\cup \{e_1\}$ when $s=5$, in particular
we picture the sets $T'$, $T''$ and the edge $e_1$ assigning a color to each of them.

\vskip 0.5truecm
\centerline{
\begin{tikzpicture}
  [scale=.2]
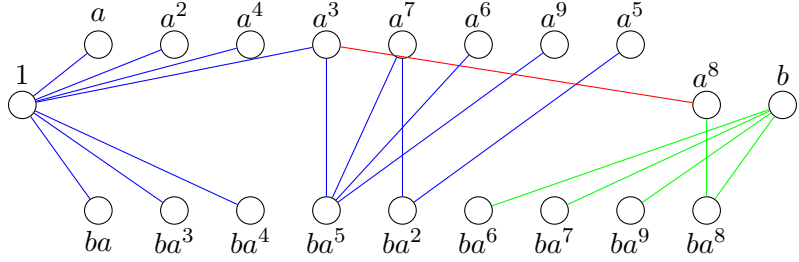
  \node (N0.1) at (25,-3) {Figure 4: $R\cup \{e_1\}$, case $s=5$. Dicyclic group of order 20.};
  \node [circle, draw](n1) at (-2,10) {};
  \node (N1) at (-2,12) {$1$};
  \node [circle, draw](n2) at (3,14)  {};
  \node (N2) at (3,16) {$a$};
  \node [circle, draw](n3) at (8,14)  {};
  \node (N3) at (8,16) {$a^2$};
  \node [circle, draw](n4) at (13,14)  {};
  \node (N4) at (13,16) {$a^4$};
  \node [circle, draw](n5) at (18,14) {};
   \node (N5) at (18,16) {$a^3$};
  \node [circle, draw](n6) at (23,14)  {};
   \node (N6) at (23,16) {$a^7$};
  \node [circle, draw](n7) at (28,14)  {};
   \node (N7) at (28,16) {$a^6$};
  \node [circle, draw](n8) at (33,14) {};
   \node (N8) at (33,16) {$a^9$};
  \node [circle, draw](n9) at (38,14)  {};
   \node (N9) at (38,16) {$a^5$};
   \node [circle, draw](n10) at (43,10)  {};
   \node (N10) at (43,12) {$a^8$};
   \node [circle, draw](n11) at (48,10)  {};
   \node (N11) at (48,12) {$b$};
   \node [circle, draw](n12) at (3,3) {};
   \node (N12) at (3,1) {$ba$};
  \node [circle, draw](n13) at (8,3)  {};
  \node (N13) at (8,1) {$ba^3$};
  \node [circle, draw](n14) at (13,3)  {};
  \node (N14) at (13,1) {$ba^4$};
  \node [circle, draw](n15) at (18,3) {};
  \node (N15) at (18,1) {$ba^5$};
  \node [circle, draw](n16) at (23,3)  {};
  \node (N16) at (23,1) {$ba^{2}$};
  \node [circle, draw](n17) at (28,3)  {};
  \node (N17) at (28,1) {$ba^6$};
  \node [circle, draw](n18) at (33,3) {};
  \node (N18) at (33,1) {$ba^7$};
  \node [circle, draw](n19) at (38,3) {};
  \node (N19) at (38,1) {$ba^9$};
  \node [circle, draw](n20) at (43,3) {};
  \node (N20) at (43,1) {$ba^8$};

\draw [blue](n1) -- (n2);
\draw [blue](n1) -- (n3);
\draw [blue](n1) -- (n4);
\draw [blue](n1) -- (n5);
\draw [blue](n1) -- (n12);
\draw [blue](n1) -- (n13);
\draw [blue](n1) -- (n14);
\draw [blue](n5) -- (n15);
\draw [blue](n6) -- (n15);
\draw [blue](n15) -- (n7);
\draw [blue](n15)--(n8);
\draw [blue](n6)--(n16);
\draw [blue](n16) -- (n9);
\draw [green](n11) -- (n17);
\draw [green](n11) -- (n18);
\draw [green](n11) -- (n19);
\draw [green](n11) -- (n20);
\draw [green](n10)--(n20);
\draw [red](n5) -- (n10);
\end{tikzpicture}
}

\section{Abelian groups with a cyclic subgroup of index $2$ and complete sets of rainbow spanning trees}\label{Ab}
Let $G$ be an abelian non-cyclic group $G$ of order $2n$, $n$ even, possesing a cyclic subgroup $H$ of index $2$.
It is well-known that $G$  is the direct product of the latter subgroup by a cyclic group, say $K$, of order $2$. ?????CITARE????
We have $G=KH$ with $K=<b>$ and $H=<a>$ and $G$ has three involutions: $b, a^{\frac{n}{2}}, ba^{\frac{n}{2}}$.

\noindent
 A starter can be constructed as follows:

\vskip0.3truecm\noindent
If $n > 4$ let
$\Sigma =\{S, S'\} \cup \{S_i, 1\le i \le \frac{n}{2}-1, \ i\ne \frac{n}{4}\} \cup \{S^*_1, S^*_2, S^*\}$.

\vskip0.3truecm\noindent
If $n=4$ let $\Sigma =\{S, S'\} \cup \{S^*_1, S^*_2, S^*\}$.

\vskip0.3truecm\noindent
With:
\vskip0.3truecm\noindent
$S=\{[a^i,a^{-i+1}],  \ 1\le i \le \frac{n}{4}\}$;  \ \ \ $S'=\{[a^i,a^{\frac{n}{2}-i}],  \ 1\le i < \frac{n}{4}\}\cup \{[1,ba^{\frac{n}{4}}\}$;

\vskip0.3truecm\noindent
$S_i=\{[1,ba^i]\}$, $1\le i \le \frac{n}{2}-1$, $i\ne \frac{n}{4}$; \ $S^*_1=\{[1,b]\}$; $S^*_2=\{[1,ba^{\frac{n}{2}}]\}$;

\vskip0.3truecm\noindent
$S^*=\{[1,a^{\frac{n}{2}}]\}$.

\vskip0.3truecm\noindent
We have:
\vskip0.3truecm\noindent
$\partial S = \{a^{2i-1},  1\le i \le \frac{n}{2}\}$; $\partial S' = \{a^{n-2r},  1\le r < \frac{n}{2}\}\cup \{ba^{\frac{n}{4}}, ba^{\frac{3n}{4}}\}$;
and both $\phi(S)$ and $\phi(S')$ are left transversal for the subgroup $I=\{1,b,a^{\frac{n}{2}}, ba^{\frac{n}{2}}\}$.

\vskip0.3truecm\noindent
$\partial S_{i}=\{ba^{i}, ba^{n-i}\}$ and $\phi(S_i)$ is left transversal for $<a>$.

\vskip0.3truecm\noindent
Finally, we have: $\partial S^*_1=\{b\}$, $\partial S^*_2=\{ba^{\frac{n}{2}}\}$, $\partial S^*=\{a^{\frac{n}{2}}\}$. Moreover, $\phi(S^*_1)= \phi(S^*_2)=\phi(S^*)=\{1\}$.

\vskip0.3truecm\noindent
With the starter above, we construct the following $1-$factors:

\vskip0.3truecm\noindent
$F_S = Orb_{I}(S)$ whose orbit under $G$ gives the $1-$factors:  $F_S, F_Sa, \dots, F_Sa^{\frac{n}{2}-1}$.

\vskip0.3truecm\noindent
$F_{S'}= Orb_{I}(S')$ whose orbit under $G$ gives the $1-$factors:  $F_{S'}, F_{S'}a, \dots, F_{S'}a^{\frac{n}{2}-1}$.

\vskip0.3truecm\noindent
$F_i = Orb_{<a>}(S_i)$ whose orbit under $G$ gives the $1-$factors: $F_i$, $F_ib$, for each $i$ with $1\le i \le \frac{n}{2}-1$, $i\ne \frac{n}{4}$.
.
\vskip0.3truecm\noindent
Finally, we have the three fixed $1-$factors $F^*_1=Orb_G([1,b])$, $F^*_2=Orb_G([1,ba^{\frac{n}{2}}])$,  $F^*=Orb_G([1,a^{\frac{n}{2}}])$.

\vskip0.3truecm\noindent
Consider the graph $R_1$ induced by the following set of edges:
$$\{[1,a^{2i-1}], [ba^{\frac{n}{4}},ba^{\frac{n}{4}+2i-1}],  1\le i \le \frac{n}{4}\}$$

\noindent
The $\frac{n}{2}$ edges of $R_1$ belong to the $\frac{n}{2}$ distinct $1-$factors $F_S$, $F_Sa$, $\dots$, $F_Sa^{\frac{n}{2}-1}$. In fact:
$[1,a^{2i-1}]=[a^i, a^{-i+1}]a^{i-1}\in F_Sa^{i-1}$ and
$[ba^{\frac{n}{4}},ba^{\frac{n}{4}+2i-1}]=[1,a^{2i-1}]ba^{\frac{n}{4}}\in F_Sa^{\frac{n}{4}+i-1}$. Moreover, for every $i$, the two edges $[1,a^{2i-1}]$ and $[ba^{\frac{n}{4}},ba^{\frac{n}{4}+2i-1}]$ have the same difference set and are in distinct orbits under $<a>$.

\vskip0.3truecm\noindent
Consider the graph $R_2$ induced by the following set of edges:
$$\{[1,a^{\frac{n}{2}-2i}], [ba^{\frac{n}{4}},ba^{\frac{3n}{4}-2i}],  1\le i < \frac{n}{4}\}\cup \{[1,ba^{\frac{n}{4}}],[ba^{\frac{n}{4}},a^{\frac{n}{2}}]\}$$

\noindent
The $\frac{n}{2}$ edges of $R_2$ belong to the $\frac{n}{2}$ distinct $1-$factors $F_{S'}$, $F_{S'}a$, $\dots$, $F_{S'}a^{\frac{n}{2}-1}$. In fact:
$[1,ba^{\frac{n}{4}}]\in F_{S'}$ and $[1,ba^{\frac{n}{4}}]ba^{\frac{n}{4}}=[ba^{\frac{n}{4}},a^{\frac{n}{2}}]\in F_{S'}a^{\frac{n}{4}}$. Moreover, for every $i$, the two edges $[1,a^{\frac{n}{2}-2i}]$ and $[ba^{\frac{n}{4}},ba^{\frac{n}{4}+2i-1}]$ have the same difference set and are in distinct orbits under $<a>$.

\noindent
$[1,a^{\frac{n}{2}-2i}]=[a^i,a^{\frac{n}{2}-i}]a^{-i}\in F_{S'}a^{\frac{n}{2}-i}$ and
$[ba^{\frac{n}{4}},ba^{\frac{3n}{4}-2i}]=[1,a^{\frac{n}{2}-2i}]ba^{\frac{n}{4}}\in F_{S'}a^{\frac{n}{4}-i}$. Moreover, for every $i$, these two edges have the same difference set and are in distinct orbits under $<a>$.

\noindent
Observe that $R_2=\{[1,ba],[ba,a^2]\}$  whenever $n=4$.

\vskip0.3truecm\noindent
If $n >4$, consider the graph $R_3$ induced by the following set of edges:
$$\{[a^{\frac{n}{4}+2}, ba^{\frac{3n}{4}+1}],[b,a^{\frac{n}{2}-1}]\}\cup \{[1,ba^i], [ba^{\frac{3n}{4}},a^{\frac{3n}{4}+i}],  1\le i \le \frac{n}{4}-1\}\cup$$
$$\cup \{[a^{\frac{n}{2}+1},ba^{\frac{3n}{4}+i+1}],[ba^{\frac{n}{2}-2i},a^{\frac{3n}{4}-i}], 1\le i \le \frac{n}{4}-2\}$$

\noindent
The $n-4$ edges of $R_3$ belong to the $n-4$ distinct $1-$factors $F_i$, $F_ib$, $1\le i \le \frac{n}{2}-1$, $i\ne \frac{n}{4}$. In fact:

\noindent
$[1,ba^i]\in F_i$ and $[ba^{\frac{3n}{4}},a^{\frac{3n}{4}+i}]\in F_i b$, $1\le i \le \frac{n}{4}-1$. Moreover, for each fixed $i$,
these two edges have the same difference set and are in distinct orbits under $<a>$.

\noindent
$[a^{\frac{n}{2}+1},ba^{\frac{3n}{4}+i+1}]\in F_{\frac{n}{4}+i}$ and $[ba^{\frac{n}{2}-2i},a^{\frac{3n}{4}-i}]\in F_{\frac{n}{4}+i}b$ , $1\le i \le \frac{n}{4}-2$. Moreover, for each fixed $i$,
these two edges have the same difference set and are in distinct orbits under $<a>$.

\noindent
$[a^{\frac{n}{4}+2}, ba^{\frac{3n}{4}+1}]\in F_{\frac{n}{2}-1}$ and $[b,a^{\frac{n}{2}-1}]\in F_{\frac{n}{2}-1}b$. Also these two edges have the same difference set and are in distinct orbits under $<a>$.

\vskip0.3truecm\noindent
Finally, if $n > 4$, let $R_4$ be the graph induced by the following two short edges: $[a^{\frac{3n}{4}}, ba^{\frac{3n}{4}}]\in F^*_1$, $[ba,a^{\frac{n}{2}+1}]\in F^*_2$, while if $n=4$ let $R_4$ be the graph induced by the following two short edges: $[a^3, ba^3]\in F^*_1$, $[b,ba^3]\in F^*_2$.

\vskip0.3truecm\noindent
If $n>4$, let  $R=R_1\cup R_2 \cup R_3 \cup R_4$, while let $R= R_1\cup R_2 \cup R_4$ whenever $n=4$. As above observed, the graph $R$  satisfies conditions (1) and (2) of Lemma \ref{L1}. Moreover, the graph $R$ has two connected components. In fact, if $n=4$ the two connected components are clearly indicated in the following figure 5. If $n >4$, a component is given by the star at $ba^{\frac{3n}{4}}$ containing all the  vertices $\{a^{\frac{3n}{4}+i}$, $0 \le i \le \frac{n}{4}-1\}$; the other component is obtained as follows: a star at $1$ containing all the vertices $\{a^i, 1\le i \le \frac{n}{2}-1\}\cup \{ba^i, 1\le i\le \frac{n}{4}\}$ plus the edge $[b, a^{\frac{n}{2}-1}]$ with $b$ of degree $1$ and the edge $[a^{\frac{n}{4}+2},ba^{\frac{3n}{4}+1}]$ with $ba^{\frac{3n}{4}+1}$ of degree 1; a star at $a^{\frac{n}{2}+1}$ containing all the vertices $\{ba^{\frac{3n}{4}+i}$, $2\le i \le \frac{n}{4}-1\}\cup \{ba\}$ and which has just the vertex $ba$ in common with the star at $1$; a star at $ba^{\frac{n}{4}}$ containing the vertex $a^{\frac{n}{2}}$ together with all the vertices $\{ba^{\frac{n}{4}+i}, 1\le i \le \frac{n}{2}-1\}$. This star has the unique vertex $ba^{\frac{n}{4}}$ in common with the star at 1 and no vertex  in common with the star at $a^{\frac{n}{2}+1}$. Finally, we have the edges $[ba^{\frac{n}{2}-2i}, a^{\frac{3n}{4}-i}]$, $1\le i \le \frac{n}{4}-2$ which are connected to the star  $ba^{\frac{n}{4}}$ and the vertices $a^{\frac{3n}{4}-i}$ have degree $1$. Therefore,  this component is a tree.

\noindent
If $n=4$, let $e_1=[a,a^3]\in F^*$ and $e_2=[b,ba^2]\in F^*$, while if $n > 4$, let $e_1=[a^{\frac{3n}{4}},a^{\frac{n}{4}}]\in F^*$ and $e_2=[ba^{\frac{3n}{4}},ba^{\frac{n}{4}}]\in F^*$.
These two edges are in distinct orbits under $<a>$ and the graphs $T_1=R\cup \{e_1\}$ and $T_2=R\cup \{e_2\}$ satisfy condition (3) of Lemma \ref{L1}. Therefore, the set ${\cal T}=\{T_1a^i \ | \ 0\le i \le s-1\}\cup \{T_2a^i \ | \ 0\le i \le s-1\}$ is a complete set of rainbow spanning trees.

\noindent
In the following Figures 5 and 6 we show $R\cup \{e_1\}$. In particular,
we picture the two connected components of R and the edge $e_1$ assigning a color to each of them.

\centerline{
\begin{tikzpicture}
  [scale=.2]
  \node (N0.1) at (85,-10) {\ \ \ \ \ \ \ \ \ \ Figure 5: \ Group $\Bbb Z_2 \times \Bbb Z_4$.};
  \node [circle, draw](n1) at (81,10) {};
  \node (N1) at (79.7,10) {$1$};
  \node [circle, draw](n2) at (86,13)  {};
  \node (N2) at (86,14.8) {$a$};
  \node [circle, draw](n3) at (91,13)  {};
  \node (N3) at (91,15.3) {$a^2$};
  \node [circle, draw](n4) at (96,13) {};
   \node (N4) at (96,15.3) {$a^3$};
  \node [circle, draw](n5) at (91,1)  {};
   \node (N5) at (89.4,1) {$b$};
   \node [circle, draw](n6) at (86,-3) {};
   \node (N6) at (86,-5.3) {$ba$};
  \node [circle, draw](n7) at (91,-3)  {};
  \node (N7) at (91,-5.3) {$ba^2$};
  \node [circle, draw](n8) at (96,-3)  {};
  \node (N8) at (96,-5.3) {$ba^3$};

\draw [blue] (n1) -- (n2);
\draw [blue] (n1) -- (n6);
\draw [blue] (n6) -- (n3);
\draw [blue] (n6) -- (n7);
\draw [green](n5) -- (n8);
\draw [green](n4) -- (n8);
\draw [red](n5) -- (n7);
\end{tikzpicture}
}

\vskip 0.5truecm
\centerline{
\begin{tikzpicture}
  [scale=.2]
  \node (N0.1) at (25,-3) { \ \ \ \ \ \ \ \ \ \ \ \ \ \ \ \ Figure 6: $R\cup \{e_1\}$ \  Group $\Bbb Z_2 \times \Bbb Z_{16}$};
  \node [circle, draw](n1) at (-2,10) {};
  \node (N1) at (-2,12) {$1$};
  \node [circle, draw](n2) at (2,17)  {};
  \node (N2) at (2,19) {$a$};
  \node [circle, draw](n3) at (6,17)  {};
  \node (N3) at (6,19) {$a^2$};
  \node [circle, draw](n4) at (10,17)  {};
  \node (N4) at (10,19) {$a^3$};
  \node [circle, draw](n5) at (14,17) {};
   \node (N5) at (14,19) {$a^4$};
  \node [circle, draw](n6) at (18,17)  {};
   \node (N6) at (18,19) {$a^5$};
  \node [circle, draw](n7) at (22,17)  {};
   \node (N7) at (22,19) {$a^6$};
  \node [circle, draw](n8) at (26,17) {};
   \node (N8) at (26,19) {$a^7$};
   \node [circle, draw](n9) at (30,17)  {};
   \node (N9) at (30,19) {$a^8$};
   \node [circle, draw](n10) at (34,17)  {};
   \node (N10) at (34,19) {$a^{10}$};
   \node [circle, draw](n11) at (38,17)  {};
   \node (N11) at (38,19) {$a^{11}$};
   \node [circle, draw](n12) at (42,17)  {};
   \node (N12) at (42,19) {$a^{9}$};
   \node [circle, draw](n13) at (50,10)  {};
   \node (N13) at (50,12) {$a^{12}$};
   \node [circle, draw](n14) at (50,17)  {};
   \node (N14) at (50,19) {$a^{13}$};
    \node [circle, draw](n15) at (54,17)  {};
   \node (N15) at (54,19) {$a^{14}$};
    \node [circle, draw](n16) at (58,9)  {};
   \node (N16) at (58,11) {$a^{15}$};
   \node [circle, draw](n17) at (-2,7.3)  {};
   \node (N17) at (-2,5.3) {$b$};
   \node [circle, draw](n18) at (2,2.5) {};
   \node (N18) at (2,0.5) {$ba$};
  \node [circle, draw](n19) at (6,2.5)  {};
  \node (N19) at (6,0.5) {$ba^{2}$};
  \node [circle, draw](n20) at (10,2.5)  {};
  \node (N20) at (10,0.5) {$ba^{3}$};
  \node [circle, draw](n21) at (17,10) {};
  \node (N21) at (19,12) {$ba^4$};
  \node [circle, draw](n22) at (18,2.5)  {};
  \node (N22) at (18,0.5) {$ba^5$};
  \node [circle, draw](n23) at (22,2.5)  {};
  \node (N23) at (22,0.5) {$ba^{6}$};
  \node [circle, draw](n24) at (26,2.5) {};
  \node (N24) at (26,0.5) {$ba^{7}$};
  \node [circle, draw](n25) at (30,2.5) {};
  \node (N25) at (30,0.5) {$ba^8$};
  \node [circle, draw](n26) at (34,2.5) {};
  \node (N26) at (34,0.5) {$ba^{9}$};
  \node [circle, draw](n27) at (38,2.5) {};
  \node (N27) at (38,0.5) {$ba^{10}$};
  \node [circle, draw](n28) at (42,2.5) {};
  \node (N28) at (42,0.5) {$ba^{11}$};
  \node [circle, draw](n29) at (46,2.5) {};
  \node (N29) at (46,0.5) {$ba^{13}$};
  \node [circle, draw](n30) at (50,2.5) {};
  \node (N30) at (50,0.5) {$ba^{14}$};
  \node [circle, draw](n31) at (54,2.5) {};
  \node (N31) at (54,0.5) {$ba^{15}$};
  \node [circle, draw](n32) at (58,2.5) {};
  \node (N32) at (58,0.5) {$ba^{12}$};
\draw [blue](n1) -- (n2);
\draw [blue](n1) -- (n3);
\draw [blue](n1) -- (n4);
\draw [blue](n1) -- (n5);
\draw [blue](n1) -- (n6);
\draw [blue](n1) -- (n7);
\draw [blue](n1) -- (n8);
\draw [blue](n1) -- (n18);
\draw [blue](n7) -- (n29);
\draw [blue](n17) -- (n8);
\draw [blue](n1) -- (n18);
\draw [blue](n17) -- (n19);
\draw [blue](n17) -- (n20);
\draw [blue](n17) -- (n21);
\draw [blue](n21) -- (n9);
\draw [blue](n21) -- (n10);
\draw [blue](n21) -- (n22);
\draw [blue](n21) -- (n23);
\draw [blue](n21) -- (n24);
\draw [blue](n21) -- (n25);
\draw [blue](n21) -- (n26);
\draw [blue](n21) -- (n27);
\draw [blue](n21) -- (n28);
\draw [blue](n23) -- (n11);
\draw [blue](n18) -- (n12);
\draw [blue](n12)--(n30);
\draw [blue](n12) -- (n31);
\draw [green](n32) -- (n16);
\draw [green](n32) -- (n13);
\draw [green](n32) -- (n14);
\draw [green](n32)-- (n15);
\draw [red](n5) -- (n13);
\end{tikzpicture}
}

\section{$2-$groups with a cyclic subgroup of index $2$ and complete sets of rainbow spanning trees }\label{2G}

\noindent
Apart from the abelian groups, the dihedral groups and the generalized
quaternion groups, for which we refer to \cite{MR}, \cite{R2} and to the
previous sections, respectively, there are two more isomorphism types
of groups of order $2n=2^{m+1}$ with a cyclic subgroup of index $2$,
see Satz 14.9 in \cite{Hu}.
In particular, it is $n \ge 8$ and they can be presented as follows, \cite[p.91]{Hu}:
\[
\begin{array}{rl}
 (i) & \quad G=\langle a,b : a^n=b^2=1, bab= a^{\frac{n}{2}-1}\rangle
\textrm{ (semidihedral group)}\\
     &                        \\
(ii) & \quad G=\langle a,b : a^n=b^2=1, bab=a^{\frac{n}{2}+1}\rangle
\end{array}
\]

\noindent
We consider these two cases separately. For each case,  we exhibit a starter, a $1-$factorization and a complete set of rainbow spanning trees.

\noindent
{\bf Case $(i)$}.

\noindent
Let $0\le r \le n-1$. Observe that $a^rb=ba^{-r}$ whenever $r$ is even, while $a^rb=ba^{\frac{n}{2}-r}$ whenever $r$ is odd. Moreover, $G$ contains exactly $\frac{n}{2}+1$ involutions: $a^{\frac{n}{2}}$ and $ba^r$, with $r$ even and $0\le r \le n-2$.

\noindent
A starter can be constructed as follows:

\vskip0.3truecm\noindent
$$\Sigma =\{S\} \cup \{S_{2t+1}, 0\le t \le \frac{n}{4}-1\} \cup \{S_{2s}, 0\le s \le \frac{n}{2}-1\}\cup$$

$$\cup \{S'_{2r+1}, 0\le r \le \frac{n}{4}-1, r\ne \frac{n}{8}\}\cup \{S^*\}$$.

\noindent
With:
\vskip0.3truecm\noindent
$S=\{[a^t,a^{-t}],  \ 1\le t \le \frac{n}{4}-1\}\cup \{[1,ba^{\frac{n}{4}+1}]\}$.

\vskip0.3truecm\noindent
$S_{2t+1}=\{[1,a^{2t+1}]\}$, $0\le t \le \frac{n}{4}-1$.

\vskip0.3truecm\noindent
$S_{2s}=\{[1,ba^{2s}]\}$, $0\le s\le \frac{n}{2}-1$.

\vskip0.3truecm\noindent
$S'_{2r+1}=\{[1,ba^{2r+1}]\}$, $0\le r \le \frac{n}{4}-1$, $r\ne \frac{n}{8}$.

\vskip0.3truecm\noindent
$S^*=\{[1,a^{\frac{n}{2}}]\}$.

\vskip0.3truecm\noindent
We have:
\vskip0.3truecm\noindent
$\partial S = \{a^{2t}, a^{-2t},  1\le t \le \frac{n}{4}-1\}\cup \{ba^{\frac{n}{4}+1}, ba^{-\frac{n}{4}+1}\}$
and $\phi(S)$ is a left transversal for the subgroup $<ba>=\{1, ba, a^{\frac{n}{2}}, ba^{\frac{n}{2}+1}\}$.

\vskip0.3truecm\noindent
For each $t$, $0\le t \le \frac{n}{4}-1$, we have: $\partial S_{2t+1}=\{a^{2t+1}, a^{-2t-1}\}$ and $\phi(S_{2t+1})$ is a left transversal for the subgroup $<a^2,b>=\{a^{2m}, ba^{2m}, 1\le m \le \frac{n}{2}\}$.

\vskip0.3truecm\noindent
For each $s$, $0\le s\le \frac{n}{2}-1$,  we have: $\partial S_{2s}=\{ba^{2s}\}$ and  $\phi(S_{2s})=\{1\}$.

\vskip0.3truecm\noindent
For each $r$, $0\le r \le \frac{n}{4}-1$, $r\ne \frac{n}{8}$, we have: $\partial S'_{2r+1}=\{ba^{2r+1}, ba^{\frac{n}{2}+2r+1}\}$ and $\phi(S'_{2r+1})$ is a left transversal for the subgroup $<a>$.

\vskip0.3truecm\noindent
Finally, we have $\partial S^*=\{a^{\frac{n}{2}}\}$ and and $\phi(S^*)=\{1\}$.

\vskip0.3truecm\noindent
With the starter above, we construct the following $1-$factors:

\vskip0.3truecm\noindent
$F = Orb_{<ba>}(S)$ whose orbit under $G$ gives the  $1-$factors:  $F, Fa, \dots, Fa^{\frac{n}{2}-1}$.

\vskip0.3truecm\noindent
For each $t$, $0\le t \le \frac{n}{4}-1$,  we obtain the $1-$factors $F_{2t+1}$ and $F_{2t+1}a$, with $F_{2t+1}=Orb_{<a^2,b>}[1,a^{2t+1}]$.

\vskip0.3truecm\noindent
For each $s$, $0\le s\le \frac{n}{2}-1$, we have the fixed $1-$factor $F_{2s}=Orb_G([1,ba^{2s}])$.

\vskip0.3truecm\noindent
For each $r$, $0\le r \le \frac{n}{4}-1$, $r\ne \frac{n}{8}$, we obtain the $1-$factors $F'_{2r+1}$ and $F'_{2r+1}b$, with $F'_{2r+1}=Orb_{<a>}[1,ba^{2r+1}]$.

\vskip0.3truecm\noindent
We also have the fixed $1-$factor $F^*=Orb_G([1,a^{\frac{n}{2}}])$.

\vskip0.3truecm\noindent
Consider the graph $R_1$ induced by the following set of edges:
$$\{[1,ba^{\frac{n}{4}+1}], [b,a^{\frac{n}{4}-1}]\}\cup \{[1,a^{2t}], [b,ba^{\frac{n}{2}+2t}], 1\le t \le \frac{n}{4}-1\}$$

\noindent
The $\frac{n}{2}$ edges of $R_1$ belongs to the $\frac{n}{2}$ distinct $1-$factors $F$, $Fa$, $\dots$, $Fa^{\frac{n}{2}-1}$. In fact,
observe that $[1,ba^{\frac{n}{4}+1}]\in F$ and $[b,a^{\frac{n}{4}-1}]= [1,ba^{\frac{n}{4}+1}]b\in Fb=Fba^{\frac{n}{2}+1}a^{\frac{n}{2}-1}= Fa^{\frac{n}{2}-1}$. Moreover, these two edges have the same difference set and are in distinct orbits under $<a>$.

\noindent
Observe also that $[1,a^{2t}]= [a^{-t},a^t]a^t \in Fa^t$ and $[b,ba^{\frac{n}{2}+2t}]=[1,a^{-2(\frac{n}{4}+t)}]b\in Fa^{-(\frac{n}{4}+t)}b =Fa^{\frac{n}{4}+t-1}$ in fact: if $t$ is even we have $Fa^{-(\frac{n}{4}+t)}b=Fba^{\frac{n}{4}+t}=Fbaa^{\frac{n}{4}+t-1}=Fa^{\frac{n}{4}+t-1}$, while if $t$ is odd we have:
$Fa^{-(\frac{n}{4}+t)}b=Fba^{\frac{n}{2}+\frac{n}{4}+t}=Fba^{\frac{n}{2}+1}a^{\frac{n}{4}+t-1}=Fa^{\frac{n}{4}+t-1}$. Moreover, for each $t$, $1\le t\le \frac{n}{4}-1$, the two edges $[1,a^{2t}]$ and $[b,ba^{\frac{n}{2}+2t}]$ have the same difference set and they are in distinct orbits under $<a>$.

\noindent
Let $R_2$ be the graph induced by the following set of edges:
$$\{[1,a^{2t+1}], [ba,ba^{\frac{n}{2}-2t}], 0\le t \le \frac{n}{4}-1\}$$

\noindent
The $\frac{n}{2}$ edges of $R_2$ belongs to the $\frac{n}{2}$ distinct $1-$factors $F_{2t+1}$, $F_{2t+1}a$, with $0\le t\le \frac{n}{4}-1$. In fact, it is $[1,a^{2t+1}]\in F_{2t+1}$ and $[1,a^{2t+1}]ba= [ba,ba^{\frac{n}{2}-2t}]\in F_{2t+1}ba=F_{2t+1}a$. Morevore, it is $\partial [1,a^{2t+1}] = \partial [ba,ba^{\frac{n}{2}-2t}]$ and, for each $t$ with $0\le t \le \frac{n}{4}-1$, these two edges are in distinct orbits under $<a>$.

\noindent
Let $R_3$ be the graph induced by the following set of edges:
$$\{[a^{\frac{n}{2}+\frac{n}{4}-2}, ba^{\frac{n}{4}-2}]\}\cup \{[a^{\frac{n}{2}+1},ba^{2t+1}], 1\le t \le \frac{n}{2}-1\}$$

\noindent
The $\frac{n}{2}$ edges of $R_3$ are short and they belongs to the $\frac{n}{2}$ distinct fixed $1-$factors $F_{2s}$, $0 \le s \le \frac{n}{2}-1$. In fact: $\partial [a^{\frac{n}{2}+1},ba^{2t+1}] =\{ba^{2t+\frac{n}{2}}\}$ with $1\le t \le \frac{n}{2}-1$. If $1\le t \le \frac{n}{4}-1$, we have $[a^{\frac{n}{2}+1},ba^{2t+1}]\in F_{\frac{n}{2}+2t}=F_{2s}$ with $\frac{n}{4}+1\le s \le \frac{n}{2}-1$. If $\frac{n}{4}\le t \le \frac{n}{2}-1$, we have $[a^{\frac{n}{2}+1},ba^{2t+1}]\in F_{\frac{n}{2}+2t}=F_{2s}$  with $0 \le s \le \frac{n}{4}-1$.
Finally $\partial[a^{\frac{n}{2}+\frac{n}{4}-2}, ba^{\frac{n}{4}-2}]= ba^{\frac{n}{2}}$ and $[a^{\frac{n}{2}+\frac{n}{4}-2}, ba^{\frac{n}{4}-2}]\in F_{\frac{n}{2}}$.

\noindent
If $n=8$ let $R_4$ be the graph induced by the following set of edges:
$$\{[ba^7, a^6], [a^7,ba^4]\}$$
While, if $n > 8$, let $R_4$ be the graph induced by the following set of edges:
$$\{[ba^{n-1}, a^{n-2r-2}], [a^{\frac{n}{2}+2r+3},ba^{4r+4}], 0\le r \le \frac{n}{4}-2,  r \ne \frac{n}{8}\}\cup$$
$$\cup \{[ba^{n-1},a^{\frac{n}{2}}],[a^{\frac{n}{2}+\frac{n}{4}+3},ba^{\frac{n}{2}+\frac{n}{4}+2}]\}$$

\noindent
If $n=8$, we have: $\partial [ba^7, a^6] = \partial [a^7,ba^4]=\{ba,ba^5\}$, these two edges are in distinct orbits under $<a>$ with
$[ba^7, a^6] \in F'_1$ and $[a^7,ba^4]=[ba^7,a^6]a^6b\in F'_1b$ since $F'_1$ is fixed by $<a>$.

\noindent
If $n>8$, we have:

\noindent
$\partial [ba^{n-1}, a^{n-2r-2}]= \partial [a^{\frac{n}{2}+2r+3},ba^{4r+4}]=\{ba^{2r+1}, ba^{\frac{n}{2}+2r+1}\}$ for each fixed $r$, with $0\le r \le \frac{n}{4}-2$, $r\ne \frac{n}{8}$. Moreover, these two edges are in distinct orbits under $<a>$ and we have; $[ba^{n-1}, a^{n-2r-2}]\in F'_{2r+1}$ and  $[a^{\frac{n}{2}+2r+3},ba^{4r+4}]=[ba^{n-1}, a^{n-2r-2}]a^{-2r-2}b\in F'_{2r+1}b$ since $F'_{2r+1}$ is fixed by $<a>$. Moreover, we have $\partial [ba^{n-1}, a^{\frac{n}{2}}]= \partial [a^{\frac{n}{2}+\frac{n}{4}+3},ba^{\frac{n}{2}+\frac{n}{4}+2}]=\{ba^{\frac{n}{2}-1}, ba^{n-1}\}$. These two edges are in distinct orbits under $<a>$ with $[ba^{n-1}, a^{\frac{n}{2}}]\in F'_{\frac{n}{2}-1}$ and $[a^{\frac{n}{2}+\frac{n}{4}+3},ba^{\frac{n}{2}+\frac{n}{4}+2}]=[ba^{n-1}, a^{\frac{n}{2}}]a^{-\frac{n}{4}-2}b\in F'_{\frac{n}{2}-1}b$ since $F'_{\frac{n}{2}-1}$ is fixed by $<a>$.

\noindent
The graph $R=R_1\cup R_2\cup R_3 \cup R_4$ satisfies conditions (1) and (2) of Lemma \ref{L1}.

\noindent
If $n=8$, the graph $R$ has two connected components: one is given by the three edges: $[ba,ba^2], [ba,ba^4], [ba^4,a^7]$ and the other by the remaining ones. Let $e_1=[a^3,a^7]\in F^*$ and $e_2=[b,ba^4]\in F^*$,  these two edges are in distinct orbits under $<a>$ and the graphs $T_1=R\cup \{e_1\}$ and $T_2=R\cup \{e_2\}$ satisfy condition (3) of Lemma \ref{L1}. Therefore, the set ${\cal T}=\{T_1a^i \ | \ 0\le i \le 3\}\cup \{T_2a^i \ | \ 0\le i \le 3\}$ is a complete set of rainbow spanning trees.

\noindent
If $n>8$, the graph $R$ has two connected components, both without cycles. More precisely, one component, say $R'$ is given by a star at $ba$ together with the edges of the set
$\{[ba^{\frac{n}{4}-2},a^{\frac{n}{2}+\frac{n}{4}-2}]\}\cup \{[ba^{4t+4},a^{\frac{n}{2}+2t+3}] \ | \ t=0, \dots, \frac{n}{8}-1\}$.
The other component, say $R''$, is given by four stars: a star at $1$, a star at $b$, a star at $a^{\frac{n}{2}+1}$ and a star at $ba^{n-1}$. The stars at $1$ and at $b$ are connected through the unique common vertex $a^{\frac{n}{4}-1}$. Their union is connected to the star at $a^{\frac{n}{2}+1}$ through the unique common vertex $ba^{\frac{n}{4}+1}$. Finally, the union of these three stars is connected to the star at $ba^{n-1}$ through the unique common edge $[a^{\frac{n}{2}+1},ba^{n-1}]$. Both these connected components have no cycles. Let $e_1=[a^{\frac{n}{2}+\frac{n}{4}-2},a^{\frac{n}{4}-2}]\in F^*$ and and $e_2=[b,ba^{\frac{n}{2}}]\in F^*$. Observe that $a^{\frac{n}{2}+\frac{n}{4}-2}$ is a vertex of $R'$ while $a^{\frac{n}{4}-2}$ is a vertex of $R''$, in the same manner $b$ is a vertex of $R''$ while $ba^{\frac{n}{2}}$ is a vertex of $R'$. Moreover, $e_1$ and $e_2$ are in distinct orbits under $<a>$ and then $T_1=R\cup \{e_1\}$ and $T_2=R\cup \{e_2\}$ satisfy condition (3) of Lemma \ref{L1}. Now, the set ${\cal T}=\{T_1a^i \ | \ 0\le i \le \frac{n}{2}-1\}\cup \{T_2a^i \ | \ 0\le i \le \frac{n}{2}-1\}$ is a complete set of rainbow spanning trees.

In the following Figures 7 and 8 we show $R\cup \{e_1\}$ when either $n = 8$ or $n = 16$. In particular
we picture the two connected components of R and the edge $e_1$ assigning a color to each of them.

\vskip 0.5truecm
\centerline{
\begin{tikzpicture}
  [scale=.2]
  \node (N0.1) at (50,-3) { \ \ \ \ \ \ \ \ \ \ \ \ \ \ \ Figure 7: $R\cup \{e_1\}$, case (i) with $n=8$};
 \node [circle, draw](n1) at (45,10) {};
  \node (N1) at (45,12) {$1$};
  \node [circle, draw](n2) at (50,14)  {};
  \node (N2) at (50,16) {$a$};
  \node [circle, draw](n3) at (55,14)  {};
  \node (N3) at (55,16) {$a^2$};
  \node [circle, draw](n4) at (60,14)  {};
  \node (N4) at (60,16) {$a^4$};
  \node [circle, draw](n5) at (65,14) {};
   \node (N5) at (65,16) {$a^5$};
  \node [circle, draw](n6) at (68,14)  {};
   \node (N6) at (68,16) {$a^6$};
  \node [circle, draw](n7) at (73.2,10)  {};
   \node (N7) at (73.2,12) {$a^3$};
  \node [circle, draw](n8) at (78,14) {};
   \node (N8) at (79,16) {$a^7$};
  \node [circle, draw](n9) at (45,7)  {};
   \node (N9) at (45,5) {$b$};
   \node [circle, draw](n10) at (48,3) {};
   \node (N10) at (48,0) {$ba^6$};
  \node [circle, draw](n11) at (53,3)  {};
  \node (N11) at (53,0) {$ba^3$};
  \node [circle, draw](n12) at (58,3)  {};
  \node (N12) at (58,0) {$ba^5$};
  \node [circle, draw](n13) at (63,3) {};
  \node (N13) at (63,0) {$ba^7$};
  \node [circle, draw](n14) at (68,3)  {};
  \node (N14) at (68,0) {$ba^2$};
  \node [circle, draw](n15) at (73,6.8)  {};
  \node (N15) at (70.5,6.8) {$ba$};
  \node [circle, draw](n16) at (78,3) {};
  \node (N16) at (78,0) {$ba^4$};

\draw [blue](n1) -- (n2);
\draw [blue](n1) -- (n3);
\draw [blue](n1) -- (n7);
\draw [blue](n1) -- (n11);
\draw [blue](n9) -- (n2);
\draw [blue](n9) -- (n4);
\draw [blue](n9) -- (n10);
\draw [blue](n11) -- (n5);
\draw [blue](n5)--(n12);
\draw [blue](n5) -- (n13);
\draw [blue](n6) -- (n13);
\draw [green](n14) -- (n15);
\draw [green](n15) -- (n16);
\draw [green](n8) -- (n16);
\draw [red](n7) -- (n8);
\end{tikzpicture}
}

\vskip 1truecm
\centerline{
\begin{tikzpicture}
  [scale=.2]
  \node (N0.1) at (25,-3) {Figure 8: $R\cup \{e_1\}$, case $(i)$ with $n =16$};
  \node [circle, draw](n1) at (-2,10) {};
  \node (N1) at (-2,12) {$1$};
  \node [circle, draw](n2) at (2,17)  {};
  \node (N2) at (2,19) {$a$};
  \node [circle, draw](n3) at (6,17)  {};
  \node (N3) at (6,19) {$a^2$};
  \node [circle, draw](n4) at (10,17)  {};
  \node (N4) at (10,19) {$a^3$};
  \node [circle, draw](n5) at (14,17) {};
   \node (N5) at (14,19) {$a^4$};
  \node [circle, draw](n6) at (18,17)  {};
   \node (N6) at (18,19) {$a^5$};
  \node [circle, draw](n7) at (22,17)  {};
   \node (N7) at (22,19) {$a^6$};
  \node [circle, draw](n8) at (26,17) {};
   \node (N8) at (26,19) {$a^7$};
  \node [circle, draw](n9) at (30,14.5)  {};
   \node (N9) at (30,16.5) {$a^9$};
   \node [circle, draw](n10) at (34,14.5)  {};
   \node (N10) at (34,16.5) {$a^8$};
   \node [circle, draw](n11) at (38,14.5)  {};
   \node (N11) at (38,16.5) {$a^{12}$};
   \node [circle, draw](n12) at (42,14.5)  {};
   \node (N12) at (42,16.5) {$a^{14}$};
   \node [circle, draw](n13) at (46,14.5)  {};
   \node (N13) at (46,16.5) {$a^{15}$};
   \node [circle, draw](n14) at (42,6)  {};
   \node (N14) at (43,8) {$a^{10}$};
    \node [circle, draw](n15) at (54,14.5)  {};
   \node (N15) at (54,16.5) {$a^{11}$};
    \node [circle, draw](n16) at (58,14.5)  {};
   \node (N16) at (58,16.5) {$a^{13}$};
    \node [circle, draw](n17) at (-2,7)  {};
   \node (N17) at (-2,5) {$b$};
   \node [circle, draw](n18) at (2,3) {};
   \node (N18) at (2,1) {$ba^{10}$};
  \node [circle, draw](n19) at (6,3)  {};
  \node (N19) at (6,1) {$ba^{12}$};
  \node [circle, draw](n20) at (10,3)  {};
  \node (N20) at (10,1) {$ba^{14}$};
  \node [circle, draw](n21) at (14,3) {};
  \node (N21) at (14,1) {$ba^5$};
  \node [circle, draw](n22) at (18,3)  {};
  \node (N22) at (18,1) {$ba^3$};
  \node [circle, draw](n23) at (22,3)  {};
  \node (N23) at (22,1) {$ba^7$};
  \node [circle, draw](n24) at (26,3) {};
  \node (N24) at (26,1) {$ba^9$};
  \node [circle, draw](n25) at (30,3) {};
  \node (N25) at (30,1) {$ba^{11}$};
  \node [circle, draw](n26) at (34,3) {};
  \node (N26) at (34,1) {$ba^{13}$};
  \node [circle, draw](n27) at (38,3) {};
  \node (N27) at (38,1) {$ba^{15}$};
  \node [circle, draw](n28) at (42,3) {};
  \node (N28) at (42,1) {$ba^2$};
  \node [circle, draw](n29) at (46,3) {};
  \node (N29) at (46,1) {$ba^4$};
  \node [circle, draw](n30) at (50,3) {};
  \node (N30) at (50,1) {$ba^6$};
  \node [circle, draw](n31) at (54,3) {};
  \node (N31) at (54,1) {$ba^8$};
  \node [circle, draw](n32) at (58,7) {};
  \node (N32) at (58,5) {$ba$};
\draw [blue](n1) -- (n2);
\draw [blue](n1) -- (n3);
\draw [blue](n1) -- (n4);
\draw [blue](n1) -- (n5);
\draw [blue](n1) -- (n6);
\draw [blue](n1) -- (n7);
\draw [blue](n1) -- (n8);
\draw [blue](n1) -- (n21);
\draw [blue](n17) -- (n4);
\draw [blue](n17) -- (n18);
\draw [blue](n17) -- (n19);
\draw [blue](n17) -- (n20);
\draw [blue](n9) -- (n21);
\draw [blue](n9) -- (n22);
\draw [blue](n9) -- (n23);
\draw [blue](n9) -- (n24);
\draw [blue](n9) -- (n25);
\draw [blue](n9) -- (n26);
\draw [blue](n9) -- (n27);
\draw [blue](n27) -- (n10);
\draw [blue](n27) -- (n11);
\draw [blue](n27) -- (n12);
\draw [blue] (n20)--(n13);
\draw [green](n14) -- (n28);
\draw [green](n15) -- (n29);
\draw [green](n16) -- (n31);
\draw [green](n32) -- (n28);
\draw [green](n32) -- (n29);
\draw [green](n32)-- (n30);
\draw [green](n32) -- (n31);
\draw [red](n3) -- (n14);
\end{tikzpicture}
}

\vskip 1truecm\noindent
{\bf Case (ii)}

\noindent
Let $0\le r \le n-1$. Observe that $a^rb=ba^r$ whenever $r$ is even, while $a^rb=ba^{\frac{n}{2}+r}$ whenever $r$ is odd. Moreover, $G$ contains exactly $3$ involutions: $a^{\frac{n}{2}}$, $b$ and $ba^{\frac{n}{2}}$.

\noindent
Let $n > 8$. A starter can be constructed as follows:

\vskip0.3truecm\noindent
$\Sigma =\{S\} \cup \{S_{2t+1}, 0\le t \le \frac{n}{8}-1$ and $\frac{n}{4}\le t \le \frac{n}{4}+\frac{n}{8}-1\} \cup \{S_{2s}, 1\le s \le \frac{n}{4}-1, s\ne \frac{n}{8}\}\cup \{S^*_1, S^*_2, S^*\}$

\vskip0.3truecm\noindent
With:
\vskip0.3truecm\noindent
$S=\{[a^t,a^{\frac{n}{2}-t-1}],  \ 0\le t \le \frac{n}{4}-1\}\cup \{[a^{\frac{n}{2}+s},a^{n-s}], \ 1\le s \le\frac{n}{4}-1 \}
\cup \{[a^{\frac{n}{2}+\frac{n}{4}},ba^{\frac{n}{2}}]\}$.

\vskip0.3truecm\noindent
$S_{2t+1}=\{[1,ba^{2t+1}]\}$, $0\le t \le \frac{n}{8}-1$ and $\frac{n}{4}\le t \le \frac{n}{4}+\frac{n}{8}-1$.

\vskip0.3truecm\noindent
$S_{2s}=\{[1,ba^{2s}]\}$, $1\le s \le \frac{n}{4}-1, s\ne \frac{n}{8}$.

\vskip0.3truecm\noindent
$S^*_1=\{[1,b]\}$ \ $S^*_2=\{[1,ba^{\frac{n}{2}}]\}$
$S^*=\{[1,a^{\frac{n}{2}}]\}$.

\vskip0.3truecm\noindent
We have:
\vskip0.3truecm\noindent
$\partial S = \{a^{t},  1\le t \le n-1, t\ne \frac{n}{2}\}\cup \{ba^{\frac{n}{2}+\frac{n}{4}}, ba^{\frac{n}{4}}\}$
and $\phi(S_1)=n$ is a left transversal for the subgroup $<b>=\{1, b\}$.

\vskip0.3truecm\noindent
$\partial S_{2t+1}=\{ba^{2t+1}, ba^{\frac{n}{2}-2t-1}\}$  \ $\partial S_{2s}= \{ba^{2s}, ba^{n-2s}\}$
and both $\phi(S_{2t+1})$ and $\phi(S_{2s})$ are both left transversal for $<a>$.

\vskip0.3truecm\noindent
Finally, we have: $\partial S^*_1=\{b\}$, $\partial S^*_2=\{ba^{\frac{n}{2}}\}$, $\partial S^*=\{a^{\frac{n}{2}}\}$. Moreover, $\phi(S^*_1)= \phi(S^*_2)=\phi(S^*)=\{1\}$.

\vskip0.3truecm\noindent
With the starter above, we construct the following $1-$factors:

\vskip0.3truecm\noindent
$F = Orb_{<b>}(S)$ whose orbit under $G$ gives the $1-$factors:  $F, Fa, \dots, Fa^{n-1}$.

\vskip0.3truecm\noindent
$F_{2t+1} = Orb_{<a>}(S_{2t+1})$ whose orbit under $G$ gives the $1-$factors:  $F_{2t+1}$, $F_{2t+1}b$, for each $t$ with
 $0\le t \le \frac{n}{8}-1$ and $\frac{n}{4}\le t \le \frac{n}{4}+\frac{n}{8}-1$.

\vskip0.3truecm\noindent
$F_{2s} = Orb_{<a>}(S_{2s})$ whose orbit under $G$ gives the $1-$factors: $F_{2s}$, $F_{2s}b$, for each $s$ with
$1\le s \le \frac{n}{4}-1, s\ne \frac{n}{8}$.

\vskip0.3truecm\noindent
Finally, we have the three fixed $1-$factors $F^*_1=Orb_G([1,b])$, $F^*_2=Orb_G([1,ba^{\frac{n}{2}}])$,  $F^*=Orb_G([1,a^{\frac{n}{2}}])$.

\vskip0.3truecm\noindent
Consider the tree $R_1$ induced by the following set of edges:
$$\{[b,a^{\frac{n}{4}}], [ba^{\frac{n}{2}},a^{\frac{n}{4}}]\}\cup \{[1,a^{\frac{n}{2}-2s}], [ba^{\frac{n}{4}},ba^{\frac{n}{4}+\frac{n}{2}-2s}], 1\le s \le \frac{n}{4}-1\}\cup$$

$$\cup \{[1,a^{\frac{n}{2}-2t-1}],[ba^{\frac{n}{4}}, ba^{\frac{n}{4}-2t-1}], 0\le t \le \frac{n}{4}-1\}$$

\noindent
The $n$ edges of $R_1$ belongs to the $n$ distinct $1-$factors $F$, $Fa$, $\dots$, $Fa^{n-1}$. In fact:
$[b,a^{\frac{n}{4}}]=[ba^{\frac{n}{2}}, a^{\frac{n}{2}+\frac{n}{4}}]a^{\frac{n}{2}}\in Fa^{\frac{n}{2}}$ and
$[ba^{\frac{n}{2}},a^{\frac{n}{4}}]=[a^{\frac{n}{4}},b]ba^{\frac{n}{4}}\in Fa^{\frac{n}{2}+\frac{n}{4}}$. Moreover, these two edges have the same difference set and are in distinct orbits under $<a>$.

\noindent
Observe that $[1,a^{\frac{n}{2}-2s}]= [a^{\frac{n}{2}+s},a^{n-s}]a^{\frac{n}{2}-s} \in Fa^{\frac{n}{2}-s}$ and
$[ba^{\frac{n}{4}}, ba^{\frac{n}{2}+\frac{n}{4}-2s}]\in Fa^{\frac{n}{2}-s}ba^{\frac{n}{4}}$ which is either $Fa^{\frac{n}{2}+\frac{n}{4}-s}$ or $Fa^{\frac{n}{4}-s}$ according to whether $s$ is even or odd.  Moreover, for each $s$, $1\le s\le \frac{n}{4}-1$, the two edges $[1,a^{\frac{n}{2}-2s}]$ and $[ba^{\frac{n}{4}}, ba^{\frac{n}{2}+\frac{n}{4}-2s}]$ have the same difference set and they are in distinct orbits under $<a>$.

\noindent
Finally, observe that $[1,a^{\frac{n}{2}-2t-1}]= [a^t,a^{\frac{n}{2}-t-1}]\in Fa^{-t}$ and
$[ba^{\frac{n}{4}}, ba^{\frac{n}{4}-2t-1}]=[1,a^{\frac{n}{2}-2t-1}]ba^{\frac{n}{4}}\in Fa^{-t}ba^{\frac{n}{4}}$
which is either $Fa^{\frac{n}{4}-t}$ or $Fa^{\frac{n}{2}+\frac{n}{4}-t}$ according to whether $t$ is even or odd.
Moreover, for each $t$, $0\le t\le \frac{n}{4}-1$, the two edges $[1,a^{\frac{n}{2}-2t-1}]$ and $[ba^{\frac{n}{4}}, ba^{\frac{n}{4}-2t-1}]$  have the same difference set and they are in distinct orbits under $<a>$.

\noindent
Let $R_2$ be the union of the two stars induced by the following set of edges:

$$\{[ba^{\frac{n}{2}}, a^{\frac{n}{2}+\frac{n}{4}+2i}], [ba^{\frac{n}{2}}, a^{\frac{n}{2}+2i}], [a^{\frac{n}{2}+\frac{n}{4}}, ba^{2i}], [a^{\frac{n}{2}+\frac{n}{4}}, ba^{\frac{n}{2}+\frac{n}{4}+2i}], 1\le i \le \frac{n}{8}-1\}$$.

\noindent
The $\frac{n}{2}-4$ edges of $R_2$ belongs to the distinct $1-$factors $F_{2s}$, $F_{2s}b$, $1\le s \le \frac{n}{4}-1$,
$s\ne \frac{n}{8}$.  In fact, for each $1\le i \le \frac{n}{8}-1$, we have: $[ba^{\frac{n}{2}}, a^{\frac{n}{2}+\frac{n}{4}+2i}]\in F_{\frac{n}{4}+2i}b$,
$[ba^{\frac{n}{2}}, a^{\frac{n}{2}+2i}]\in F_{2i}b$, $[a^{\frac{n}{2}+\frac{n}{4}}, ba^{2i}]\in F_{\frac{n}{4}+2i}$, $[a^{\frac{n}{2}+\frac{n}{4}}, ba^{\frac{n}{2}+\frac{n}{4}+2i}]\in F_{2i}$.
Moreover, for each $i$ we have:
$\partial [ba^{\frac{n}{2}}, a^{\frac{n}{2}+2i}]= \partial [a^{\frac{n}{2}+\frac{n}{4}}, ba^{\frac{n}{2}+\frac{n}{4}+2i}]$ and
$\partial [ba^{\frac{n}{2}}, a^{\frac{n}{2}+\frac{n}{4}+2i}]= \partial [a^{\frac{n}{2}+\frac{n}{4}}, ba^{2i}]$ and the edges with the same difference set are in distinct orbits under $<a>$.

\noindent
Let $R_3$ be the union of the three stars induced by the following set of edges:

$$\{[a^{\frac{n}{2}}, ba^{\frac{n}{2}+2i+1}], [b, a^{\frac{n}{2}+2i+1}], [ba^{\frac{n}{2}}, a^{\frac{n}{2}+\frac{n}{4}+2i+1}], [a^{\frac{n}{2}}, ba^{\frac{n}{4}+2i+1}], 0\le i \le \frac{n}{8}-1\}$$.
The $\frac{n}{2}$ edges of $R_3$ belongs to the distinct $1-$factors: $F_{2t+1}$, $F_{2t+1}b$, with
$0\le t \le \frac{n}{8}-1$ and $\frac{n}{4}\le t \le \frac{n}{4}+\frac{n}{8}-1$. In fact:
$[a^{\frac{n}{2}}, ba^{\frac{n}{2}+2i+1}]\in F_{2i+1}$,  $[b, a^{\frac{n}{2}+2i+1}]\in F_{2i+1}b$.
Also, $[ba^{\frac{n}{2}}, a^{\frac{n}{2}+\frac{n}{4}+2i+1}]\in F_{\frac{n}{2}+\frac{n}{4}-2i-1}$,
in fact $[ba^{\frac{n}{2}}, a^{\frac{n}{2}+\frac{n}{4}+2i+1}] = [1, ba^{\frac{n}{2}+\frac{n}{4}-2i-1}]a^{-\frac{n}{4}+2i+1}$, and
$[a^{\frac{n}{2}}, ba^{\frac{n}{4}+2i+1}]\in F_{\frac{n}{2}+\frac{n}{4}-2i-1}b$. Moreover we have:
$\partial [a^{\frac{n}{2}}, ba^{\frac{n}{2}+2i+1}]=\partial  [b, a^{\frac{n}{2}+2i+1}]$ and
$\partial [ba^{\frac{n}{2}}, a^{\frac{n}{2}+\frac{n}{4}+2i+1}]= [a^{\frac{n}{2}}, ba^{\frac{n}{4}+2i+1}]$ and edges with the same difference set are in distinct orbits under $<a>$.

\noindent
Finally, let $R_4$ be inuduced by the two edges $[a^{\frac{n}{2}+\frac{n}{4}}, ba^{\frac{n}{2}+\frac{n}{4}}]\in F^*_1$ and $[b,a^{\frac{n}{2}}]\in F^*_2$ and  which are both short.

The graph $R=R_1\cup R_2 \cup R_3 \cup R_4$  satisfies conditions (1) and (2) of Lemma \ref{L1}.  It has two connected components. One is given by the union of 5 stars: a star at $1$ and a star at $ba^{\frac{n}{4}}$ without common vertices and connected through the unique edge $[ba^{\frac{n}{2}}, a^{\frac{n}{4}}]$; a star at $ba^{\frac{n}{2}}$ with just the two vertices $ba^{\frac{n}{2}}, a^{\frac{n}{4}}$ in common with the previous two stars, a star at $b$ which is connected to the previous three stars through the unique edge  $[b, a^{\frac{n}{4}}]$ and a star at $a^{\frac{n}{2}}$ connected to the previous four stars through the unique edge $[b,a^{\frac{n}{2}}]$. The other component of $R$ is a star at $a^{\frac{n}{2}+\frac{n}{4}}$.

\noindent
Let $e_1=[a^{\frac{n}{2}+\frac{n}{4}},a^{\frac{n}{4}}]\in F^*$ and $e_2=[ba^{\frac{n}{2}+\frac{n}{4}},ba^{\frac{n}{4}}]\in F^*$,
these two edges are in distinct orbits under $<a>$ and the graphs $T_1=R\cup \{e_1\}$ and $T_2=R\cup \{e_2\}$ satisfy condition (3) of Lemma \ref{L1}. Therefore, the set ${\cal T}=\{T_1a^i \ | \ 0\le i \le \frac{n}{2}-1\}\cup \{T_2a^i \ | \ 0\le i \le \frac{n}{2}-1\}$ is a complete set of rainbow spanning trees.

\noindent
If $n=8$, a starter is given by:

\vskip0.3truecm\noindent
$\Sigma =\{S\} \cup \{S_{2t+1}, 0\le t \le \frac{n}{8}-1$ and $\frac{n}{4}\le t \le \frac{n}{4}+\frac{n}{8}-1\} \cup \{S^*_1, S^*_2, S^*\}$.

\par
We have  the $1-$factors:

\vskip0.3truecm\noindent
$F = Orb_{<b>}(S)$ whose orbit under $G$ gives the $1-$factors:  $F, Fa, \dots, Fa^{n-1}$.

\vskip0.3truecm\noindent
$F_{2t+1} = Orb_{<a>}(S_{2t+1})$ whose orbit under $G$ gives the $1-$factors:  $F_{2t+1}$, $F_{2t+1}b$, for each $t$ with
 $0\le t \le \frac{n}{8}-1$ and $\frac{n}{4}\le t \le \frac{n}{4}+\frac{n}{8}-1$.

 \vskip0.3truecm\noindent
$F^*_1=Orb_G([1,b])$, $F^*_2=Orb_G([1,ba^{\frac{n}{2}}])$,  $F^*=Orb_G([1,a^{\frac{n}{2}}])$.

Then, we repeat the same construction above with the graph $R=R_1 \cup R_3 \cup R_4$.

In the following Figures 9 and 10 we show $R\cup \{e_1\}$. In particular
we picture the two connected components of R and the edge $e_1$ assigning a color to each of them.

\vskip 0.5truecm
\centerline{
\begin{tikzpicture}
  [scale=.2]
  \node (N0.1) at (25,-3) {Figure 9: $R\cup \{e_1\}$, case $(ii)$ with $n =16$};
  \node [circle, draw](n1) at (-2,10) {};
  \node (N1) at (-2,12) {$1$};
  \node [circle, draw](n2) at (2,17)  {};
  \node (N2) at (2,19) {$a$};
  \node [circle, draw](n3) at (6,17)  {};
  \node (N3) at (6,19) {$a^2$};
  \node [circle, draw](n4) at (10,17)  {};
  \node (N4) at (10,19) {$a^3$};
  \node [circle, draw](n5) at (14,17) {};
   \node (N5) at (14,19) {$a^4$};
  \node [circle, draw](n6) at (18,17)  {};
   \node (N6) at (18,19) {$a^5$};
  \node [circle, draw](n7) at (22,17)  {};
   \node (N7) at (22,19) {$a^6$};
  \node [circle, draw](n8) at (26,17) {};
   \node (N8) at (26,19) {$a^7$};
   \node [circle, draw](n9) at (30,17)  {};
   \node (N9) at (30,19) {$a^9$};
   \node [circle, draw](n10) at (34,17)  {};
   \node (N10) at (34,19) {$a^{11}$};
   \node [circle, draw](n11) at (38,17)  {};
   \node (N11) at (38,19) {$a^{13}$};
   \node [circle, draw](n12) at (42,17)  {};
   \node (N12) at (42,19) {$a^{15}$};
   \node [circle, draw](n13) at (46,17)  {};
   \node (N13) at (46,19) {$a^{14}$};
   \node [circle, draw](n14) at (50,17)  {};
   \node (N14) at (50,19) {$a^{10}$};
    \node [circle, draw](n15) at (54,17)  {};
   \node (N15) at (54,19) {$a^{8}$};
    \node [circle, draw](n16) at (58,9)  {};
   \node (N16) at (58,11) {$a^{12}$};
   \node [circle, draw](n17) at (-2,7.3)  {};
   \node (N17) at (-2,5.3) {$ba^4$};
   \node [circle, draw](n18) at (2,2.5) {};
   \node (N18) at (2,0.5) {$ba^{6}$};
  \node [circle, draw](n19) at (6,2.5)  {};
  \node (N19) at (6,0.5) {$ba^{8}$};
  \node [circle, draw](n20) at (10,2.5)  {};
  \node (N20) at (10,0.5) {$ba^{10}$};
  \node [circle, draw](n21) at (14,2.5) {};
  \node (N21) at (14,0.5) {$ba$};
  \node [circle, draw](n22) at (18,2.5)  {};
  \node (N22) at (18,0.5) {$ba^3$};
  \node [circle, draw](n23) at (22,2.5)  {};
  \node (N23) at (22,0.5) {$ba^{13}$};
  \node [circle, draw](n24) at (26,2.5) {};
  \node (N24) at (26,0.5) {$ba^{15}$};
  \node [circle, draw](n25) at (30,2.5) {};
  \node (N25) at (30,0.5) {$b$};
  \node [circle, draw](n26) at (34,2.5) {};
  \node (N26) at (34,0.5) {$ba^{9}$};
  \node [circle, draw](n27) at (38,2.5) {};
  \node (N27) at (38,0.5) {$ba^{11}$};
  \node [circle, draw](n28) at (42,2.5) {};
  \node (N28) at (42,0.5) {$ba^5$};
  \node [circle, draw](n29) at (46,2.5) {};
  \node (N29) at (46,0.5) {$ba^7$};
  \node [circle, draw](n30) at (50,2.5) {};
  \node (N30) at (50,0.5) {$ba^2$};
  \node [circle, draw](n31) at (54,2.5) {};
  \node (N31) at (54,0.5) {$ba^{12}$};
  \node [circle, draw](n32) at (58,2.5) {};
  \node (N32) at (58,0.5) {$ba^{14}$};
\draw [blue](n1) -- (n2);
\draw [blue](n1) -- (n3);
\draw [blue](n1) -- (n4);
\draw [blue](n1) -- (n5);
\draw [blue](n1) -- (n6);
\draw [blue](n1) -- (n7);
\draw [blue](n1) -- (n8);
\draw [blue](n5) -- (n19);
\draw [blue](n17) -- (n18);
\draw [blue](n17) -- (n19);
\draw [blue](n17) -- (n20);
\draw [blue](n17) -- (n21);
\draw [blue](n17) -- (n22);
\draw [blue](n17) -- (n23);
\draw [blue](n17) -- (n24);
\draw [blue](n5) -- (n25);
\draw [blue](n19) -- (n11);
\draw [blue](n19) -- (n12);
\draw [blue](n19) -- (n13);
\draw [blue](n19) -- (n14);
\draw [blue](n25) -- (n9);
\draw [blue](n25) -- (n10);
\draw [blue](n25) -- (n15);
\draw [blue](n15)--(n26);
\draw [blue](n15) -- (n27);
\draw [blue](n15) -- (n28);
\draw [blue](n15) -- (n29);

\draw [green](n16) -- (n30);
\draw [green](n16) -- (n31);
\draw [green](n16)-- (n32);
\draw [red](n16) -- (n5);
\end{tikzpicture}
}

\vskip 1truecm
\centerline{
\begin{tikzpicture}
  [scale=.2]
  \node (N0.1) at (50,-3) { \ \ \ \ \ \ \ \ \ \ \ \ \ \ \ Figure 10: $R\cup \{e_1\}$, case (ii) with $n=8$};
 \node [circle, draw](n1) at (45,10) {};
  \node (N1) at (45,12) {$1$};
  \node [circle, draw](n2) at (50,14)  {};
  \node (N2) at (50,16) {$a$};
  \node [circle, draw](n3) at (55,14)  {};
  \node (N3) at (55,16) {$a^2$};
  \node [circle, draw](n4) at (60,14)  {};
  \node (N4) at (60,16) {$a^3$};
  \node [circle, draw](n5) at (65,14) {};
   \node (N5) at (65,16) {$a^7$};
  \node [circle, draw](n6) at (70,14)  {};
   \node (N6) at (70,16) {$a^5$};
  \node [circle, draw](n7) at (75,14)  {};
   \node (N7) at (75,16) {$a^4$};
  \node [circle, draw](n8) at (80,8.5) {};
   \node (N8) at (80,10.5) {$a^6$};
  \node [circle, draw](n9) at (45,7)  {};
   \node (N9) at (45,5) {$ba^2$};
   \node [circle, draw](n10) at (50,3) {};
   \node (N10) at (50,0) {$ba^4$};
  \node [circle, draw](n11) at (55,3)  {};
  \node (N11) at (55,0) {$ba$};
  \node [circle, draw](n12) at (60,3)  {};
  \node (N12) at (60,0) {$ba^7$};
  \node [circle, draw](n13) at (65,3) {};
  \node (N13) at (65,0) {$b$};
  \node [circle, draw](n14) at (70,3)  {};
  \node (N14) at (70,0) {$ba^5$};
  \node [circle, draw](n15) at (75,3)  {};
  \node (N15) at (75,0) {$ba^3$};
  \node [circle, draw](n16) at (80,3) {};
  \node (N16) at (80,0) {$ba^6$};

\draw [blue](n1) -- (n2);
\draw [blue](n1) -- (n3);
\draw [blue](n1) -- (n4);
\draw [blue](n9) -- (n10);
\draw [blue](n9) -- (n11);
\draw [blue](n9) -- (n12);
\draw [blue](n10) -- (n3);
\draw [blue](n10) -- (n5);
\draw [blue](n3)--(n13);
\draw [blue](n13) -- (n6);
\draw [blue](n13) -- (n7);
\draw [blue](n14) -- (n7);
\draw [blue](n15) -- (n7);
\draw [green](n8) -- (n16);
\draw [red](n3) -- (n8);
\end{tikzpicture}
}

\end{document}